\documentclass{article}
\usepackage{pinlabel}
\usepackage{amssymb,amsmath,amsthm,amsfonts,fancyhdr}
\usepackage{charter,eulervm}
\usepackage{a4wide}
\usepackage[english]{babel}
\title{A topological set theory\\ implied by $\Th{ZF}$ and $\Th{GPK^+_\infty}$}
\author{Andreas Fackler}

\newcommand{\emphind}[1]{\index{#1}\emph{#1}}

\newtheorem{thm}{Theorem}
\newtheorem{prp}[thm]{Proposition}
\newtheorem{lem}[thm]{Lemma}
\newtheorem{cor}[thm]{Corollary}

\newcommand{\Fa}[1]{{\forall} #1 \;}
\newcommand{\Ex}[1]{{\exists} #1  \;}
\newcommand{\fTupL}[1]{\langle #1 \rangle}
\newcommand{\fTup}[2]{\fTupL{#1 , #2}}
\newcommand{\fSeq}[2]{\fTupL{#1 \:|\: #2}}
\newcommand{\rng}{\mathrm{rng}}
\newcommand{\dom}{\mathrm{dom}}

\newcommand{\fCl}{\mathrm{cl}}
\newcommand{\fTc}{\mathrm{trcl}}
\newcommand{\fInt}{\mathrm{int}}

\newcommand{\On}{\mathit{On}}

\newcommand{\Atom}{\mathbb{A}}

\newcommand{\VV}{\mathbb{V}}
\newcommand{\TT}{\mathbb{T}}

\newcommand{\calI}{\mathcal{I}}
\newcommand{\calK}{\mathcal{K}}

\newcommand{\Discrete}{\mathcal{D}}
\newcommand{\ExSpc}{\mathrm{Exp}}

\newcommand{\id}{\mathrm{id}}

\newcommand{\Th}[1]{\mathsf{#1}}

\newcommand{\Class}[2]{\{#1 \: | \: #2\}}

\newcommand{\wt}[1]{{\widetilde{#1}}}

\newenvironment{myenumerate} { \begin{enumerate} \setlength{\itemsep}{0pt}}{ \end{enumerate} }
\newenvironment{myitemize} { \begin{itemize} \setlength{\itemsep}{0pt}}{ \end{itemize} }

\begin{document}

\maketitle

\abstract{
We present a system of axioms motivated by a topological intuition: The set of subsets of any set is a topology on that set. On the one hand, this system is a common weakening of Zermelo-Fraenkel set theory $\Th{ZF}$, the positive set theory $\Th{GPK}^+_\infty$ and the theory of hyperuniverses. On the other hand, it retains most of the expressiveness of these theories and has the same consistency strength as $\Th{ZF}$. We single out the additional axiom of the universal set as the one that increases the consistency strength to that of $\Th{GPK}^+_\infty$ and explore several other axioms and interrelations between those theories. Our results are independent of whether the empty class is a set and whether atoms exist.

This article is a revised version of the first part of the author's doctoral thesis \cite{Fackler2012}.}

\section*{Introduction}
\setlength{\parindent}{0ex}
\setlength{\parskip}{0.6eM}

An axiomatic set theory can be thought of as an effort to make precise which classes are sets. It simultaneously aims at providing enough freedom of construction for all of classical mathematics and still remain consistent. It therefore must imply that all ``reasonable'' class comprehensions $\{ x \mid \phi(x)\}$ produce sets and explain why the Russell class $\{x \mid x\notin x\}$ does not. The answer given by e. g. Zermelo-Fraenkel set theory ($\Th{ZF}$) is the \emphind{Limitation of Size Principle}: Only small classes are sets. However, the totality of all mathematical objects, the \emphind{universe} $\VV = \{x \mid x{=}x\}$, is a proper class in $\Th{ZF}$.

Two very different ideas of sethood lead to a family of theories which do allow $\VV \in \VV$:

Firstly, one might blame the negation in the formula $x\notin x$ for Russell's paradox. The collection of \emphind{generalized positive formula}s is recursively defined by several construction steps not including negation. If the existence of  $\{x \mid \phi(x)\}$ is stipulated for every generalized positive formula, a beautiful ``positive'' set theory\index{positive set theory} emerges.

Secondly, instead of demanding that every class is a set, one might settle for the ability to approximate it by a least superset, a \emphind{closure} in a topological sense.

Surprisingly such ``topological'' set theories\index{topological set theory} tend to prove the comprehension principle for generalized positive formulas, and conversely, in positive set theory, the universe is a topological space. More precisely, the sets are closed with respect to intersections and finite unions, and the universe is a set itself, so the sets represent the closed subclasses of a topology on $\VV$. A class is a set if and only if it is topologically closed.

The first model of such a theory was constructed by R. J. Malitz in \cite{Malitz1976} under the condition of the existence of certain large cardinal numbers. E. Weydert, M. Forti and R. Hinnion were able to show in \cite{Weydert1989, FortiHinnion1989} that in fact a weakly compact cardinal suffices. In \cite{Esser1997} and \cite{Esser1999}, O. Esser exhaustively answered the question of consistency for a specific positive set theory, $\Th{GPK}^+_\infty$ with a choice principle, and showed that it is mutually interpretable with a variant of Kelley-Morse set theory.

\subsection*{Atoms, sets, classes and topology}

We incorporate proper classes into all the theories we consider. This enables us to write down many arguments in a more concise yet formally correct way, and helps separate the peculiarities of particular theories from the common assumptions about atoms, sets and classes.

We use the language of set theory with atoms, whose non-logical symbols are the binary relation symbol $\in$ and the constant symbol $\Atom$. We say ``$X$ is an element of $Y$'' for $X\in Y$. We call $X$ an \emph{atom}\index{atom} if $X\in\Atom$, and otherwise we call $X$ a \emph{class}\index{class}. If a class is an element of any other class, it is a \emph{set}\index{set}; otherwise it is a \emph{proper class}\index{proper class}.

We denote the objects of our theories -- all atoms, sets and classes -- by capital letters and use lowercase letters for sets and atoms only, so:
\begin{myitemize}
\item $\Fa{x} \phi(x)$ means $\Fa{X.} (\Ex{Y} X{\in} Y) \Rightarrow \phi(X)\;$ and
\item $\Ex{x} \phi(x)$ means $\Ex{X.} (\Ex{Y} X{\in} Y) \:\wedge\: \phi(X)$.
\end{myitemize}

For each formula $\phi$ let $\phi^C$ be its relativization to the sets and atoms, that is, every quantified variable in $\phi$ is replaced by a lowercase variable in $\phi^C$.

Free variables in formulas that are supposed to be sentences are implicitly universally quantified. For example, we usually omit the outer universal quantifiers in axioms.

Using these definitions and conventions, we can now state the basic axioms concerning atoms, sets and classes. Firstly, we assume that classes are uniquely defined by their extension, that is, two classes are equal iff they have the same elements. Secondly, atoms do not have any elements. Thirdly, there are at least two distinct sets or atoms. And finally, any collection of sets and atoms which can be defined in terms of sets, atoms and finitely many fixed parameters, is a class. Formally:
\begin{eqnarray*}
\text{\emph{Extensionality}\index{axiom!extensionality}} & \quad & (X,Y {\notin} \Atom \; \wedge \; \Fa{Z.} Z{\in}X \Leftrightarrow Z{\in} Y) \;\Rightarrow\; X{=} Y\\
\text{\emph{Atoms}\index{axiom!atoms}} & \quad & X\in \Atom \quad \Rightarrow \quad Y \notin X\\
\text{\emph{Nontriviality}\index{axiom!nontriviality}} & \quad & \Ex{x{,}y} \; x \neq y\\
\text{\emph{Comprehension}\index{axiom!comprehension}}(\psi) & \quad & \Ex{Z{\notin}\Atom.} \Fa{x.} x{\in} Z \Leftrightarrow \psi(x,\Vec{P}) \;\text{ for all formulas $\psi=\phi^C$}
\end{eqnarray*}
We will refer to these axioms as the \emph{class axioms}\index{class axioms} from now on. Note that the object $\Atom$ may well be a proper class, or a set. The atoms axiom implies however that $\Atom$ is not an atom.

We call the axiom scheme given in the fourth line the \emphind{weak comprehension} scheme. It can be strengthened by removing the restriction on the formula $\psi$, instead allowing $\psi$ to be any formula -- even quantifying over all classes. Let us call that variant the \emphind{strong comprehension} scheme. The axiom of extensionality implies the uniqueness of the class $Z$. We also write $\{w \mid \psi(w,\Vec{P})\}$ for $Z$, and generally use the customary notation for comprehensions, e.g. $\{x_1,\ldots, x_n\} = \{y \mid y{=}x_1 \vee \ldots \vee y{=}x_n\}$ for the class with finitely many elements $x_1, \ldots, x_n$, $\emptyset = \{w \mid w{\neq} w\}$ for the empty class and $\VV = \{w \mid w{=}w\}$ for the universal class. Also let $\TT = \{x \mid \exists y. y{\in}x\}$ be the class of nonempty sets. The weak comprehension scheme allows us to define unions, intersections and differences in the usual way.

Given the class axioms, we can now define several topological terms. They all make sense in this weak theory, but one has to carefully avoid for now the assumption that any class is a set. Also, the ``right'' definition of a topology in our context is the collection of all nonempty closed sets instead of all open sets.

For given classes $A$ and $T$, we call $A$ $T$-\emph{closed}\index{closed class} if $A=\emptyset$ or $A\in T$. A \emphind{topology} on a class $X$ is a class $T$ of nonempty subsets of $X$, such that:
\begin{myitemize}
\item $X$ is $T$-closed.
\item $\bigcap B$ is $T$-closed for every nonempty class $B{\subseteq}T$.
\item $a\cup b$ is $T$-closed for all $T$-closed sets $a$ and $b$.
\end{myitemize}
The class $X$, together with $T$, is then called a \emphind{topological space}. If $A$ is a $T$-closed class, then its complement $\complement A = X\setminus A$ is $T$-\emph{open}\index{open class}. A class which is both $T$-closed and $T$-open is $T$-\emph{clopen}. The intersection of all $T$-closed supersets of a class $A\subseteq X$ is the least $T$-closed superset and is called the $T$-\emphind{closure} $\fCl_T(A)$ of $A$. Then $\fInt_T(A) = \complement \fCl_T (\complement A)$ is the largest $T$-open subclass of $A$ and is called the $T$-\emphind{interior} of $A$. Every $A$ with $x\in \fInt_T (A)$ is a $T$-\emphind{neighborhood} of the point $x$. The explicit reference to $T$ is often omitted and $X$ itself is considered a topological space, if the topology is clear from the context.

If $S\subset T$ and both are topologies, we call $S$ \emph{coarser}\index{coarse topology} and $T$ \emph{finer}\index{fine topology}. An intersection of several topologies on a set $X$ always is a topology on $X$ itself. Thus for every class $B$ of subsets of $X$, if there is a coarsest topology $T\supseteq B$, then that is the intersection of all topologies $S$ with $B\subseteq S$. We say that $B$ is a \emphind{subbase} for $T$ and that $T$ is \emph{generated}\index{generated topology} by $B$.

If $A\subseteq X$, we call a subclass $B\subseteq A$ \emph{relatively closed}\index{relatively closed class} in $A$ if there is a $T$-closed $C$ such that $B=A \cap C$, and similarly for \emph{relatively open}\index{relatively open class} and \emph{relatively clopen}\index{relatively clopen class}. If every subclass of $A$ is relatively closed in $A$, we say that $A$ is \emph{discrete}\index{discrete class}. Thus a $T$-closed set $A$ is discrete iff all its nonempty subclasses are elements of $T$. Note that there is an equivalent definition of the discreteness of a class $A\subseteq X$ which can be expressed without quantifying over classes: $A$ is discrete iff it contains none of its accumulation points, where an \emph{accumulation point} is a point $x\in X$ which is an element of every $T$-closed $B\supseteq A{\setminus}\{x\}$. Formally, $A$ is discrete iff it has at most one point or:
\[\Fa{x{\in}A} \Ex{b{\in}T.} \; A \subseteq b{\cup}\{x\} \;\wedge\; x{\notin} b\]

\index{separation axioms}A topological space $X$ is $T_1$ if for all distinct $x,y\in X$ there exists an open $U\subseteq X$ with $y\notin U\ni x$, or equivalently, if every singleton $\{x\}\subseteq X$ is closed. $X$ is $T_2$ or \emph{Hausdorff}\index{Hausdorff space} if for all distinct $x,y\in X$ there exist disjoint open $U,V\subseteq X$ with $x\in U$ and $y\in V$. It is \emph{regular}\index{regular space} if for all closed $A\subseteq X$ and all $x\in X\setminus A$ there exist disjoint open $U,V\subseteq X$ with $A\subseteq U$ and $x\in V$. $X$ is $T_3$ if it is regular and $T_1$. It is \emph{normal}\index{normal space} if for all disjoint, closed $A,B\subseteq X$ there exist disjoint open $U,V\subseteq X$ with $A\subseteq U$ and $B\subseteq V$. $X$ is $T_4$ if it is normal and $T_1$.

A map $f:X\rightarrow Y$ between topological spaces is \emphind{continuous} if all preimages $f^{-1} [A]$ of closed sets $A\subseteq Y$ are closed, and it is \emphind{closed} if all images $f[A]$ of closed sets $A\subseteq X$ are closed.

Let $\calK$ be any class. We consider a class $A$ to be $\calK$-\emph{small} if it is empty or there is a surjection from a member of $\calK$ onto $A$, that is:
\[A=\emptyset \quad \vee \quad \Ex{x{\in}\calK}\Ex{F{:}x{\rightarrow}A} F[x]{=}A\]
Otherwise, $A$ is $\calK$-\emph{large}\index{large class}. We say $\calK$-\emphind{few} for ``a $\calK$-small collection of'', and $\calK$-\emphind{many} for ``a $\calK$-large collection of''. Although we quantified over classes in this definition, we will only use it in situations where there is an equivalent first-order formulation.

If all unions of $\calK$-small subclasses of a topology $T$ are $T$-closed, then $T$ is called \emph{$\calK$-additive}\index{additivity} or a \emph{$\calK$-topology}. If $T$ is a subclass of every $\calK$-topology $S\supseteq B$ on $X$, then $T$ is $\calK$-\emph{generated}\index{generated topology} by $B$ on $X$ and $B$ is a $\calK$-\emphind{subbase} of $T$ on $X$. If every element of $T$ is an intersection of elements of $B$, $B$ is a \emphind{base} of $T$.

A topology $T$ on $X$ is $\calK$-\emphind{compact} if every $T$-cocover has a $\calK$-small $T$-subcocover, where a $T$-\emphind{cocover} is a class $B\subseteq T$ with $\bigcap B = \emptyset$. Dually, we use the more familiar term \emphind{open cover} for a collection of $T$-open classes whose union is $X$, where applicable.

For all classes $A$ and $T$, let
\[\square_T A = \{b{\in} T \mid b {\subseteq} A\} \quad\text{ and }\quad \lozenge_T A = \{ b{\in} T \mid b \cap A \neq \emptyset\}\text{.}\]
If $T$ is a topology on $X$, and if for all $a,b\in T$ the classes $\square_T a \cap \lozenge_T b$ are sets, then the set $T = \square_T X = \lozenge_T X$, together with the topology $S$ $\calK$-generated by $\{\square_T a {\cap} \lozenge_T b \mid a,b{\in} T\}$ is called the $\calK$-\emphind{hyperspace} (or \emphind{exponential space}) of $X$ and denoted by $\ExSpc_\calK (X,T) = \fTup{\square_T X}{S}$, or in the short form: $\ExSpc_\calK (X)$. Since $\square_T a = \square_T a \cap \lozenge_T X$ and $\lozenge_T a = \square_T X \cap \lozenge_T a$, the classes $\square_T a$ and $\lozenge_T a$ are also sets and constitute another $\calK$-subbase of the exponential $\calK$-topology. A notable subspace of $\ExSpc_\calK (X)$ is the space $\ExSpc^c_\calK (X)$ of $\calK$-compact subsets. In fact, this restriction suggests the canonical definition $\ExSpc^c_\calK (f)(a) = f[a]$ of a map $\ExSpc^c_\calK (f) : \ExSpc^c_\calK (X) \rightarrow \ExSpc^c_\calK(Y)$ for every continuous $f:X\rightarrow Y$, because continuous images of $\calK$-compact sets are $\calK$-compact. Moreover, $\ExSpc^c_\calK (f)$ is continuous itself.

$\calK$ should be pictured as a cardinal number, but prior to stating the axioms of essential set theory, the theory of ordinal and cardinal numbers is not available. However, to obtain useful ordinal numbers, an axiom stating that the additivity is greater than the cardinality of any discrete set is needed. Fortunately, this can be expressed using the class $\Discrete$ of all discrete sets as the additivity.

\section{Essential Set Theory}

Consider the following, in addition to the class axioms:
\begin{eqnarray*}
\text{\emph{1st Topology Axiom}\index{axiom!topology}} & \quad & \VV\in\VV\\
\text{\emph{2nd Topology Axiom}\index{axiom!topology}} & \quad & \text{If $A{\subseteq}\TT$ is nonempty, then $\bigcap A$ is $\TT$-closed.}\\
\text{\emph{3rd Topology Axiom}\index{axiom!topology}} & \quad & \text{If $a$ and $b$ are $\TT$-closed, then $a{\cup}b$ is $\TT$-closed.}\\
T_1 & \quad & \text{$\{a\}$ is $\TT$-closed.}\\
\text{\emph{Exponential}\index{axiom!exponential}} & \quad & \text{$\square_\TT a \cap \lozenge_\TT b$ is $\TT$-closed.}\\
\text{\emph{Discrete Additivity}\index{axiom!additivity}} & \quad & \bigcup A \text{ is $\TT$-closed for every $\Discrete$-small class $A$.}
\end{eqnarray*}
We call this system of axioms \emphind{topological set theory}, or in short: $\Th{TS}$, and the theory $\Th{TS}$ without the 1st topology axiom \emphind{essential set theory} or $\Th{ES}$. We will mostly work in $\Th{ES}$ and explicitly single out the consequences of $\VV\in\VV$.

In $\Th{ES}$, the class $\TT = \square_\TT \VV = \lozenge_\TT \VV$ of all nonempty sets satisfies all the axioms of a topology on $\VV$, except that it does not need to contain $\VV$ itself. Although it is not necessarily a class, we can therefore consider the collection of $\VV$ and all nonempty sets a topology on $\VV$ and informally attribute topological notions to it. We will call it the \emphind{universal topology} and whenever no other topology is explicitly mentioned, we will refer to it. Since no more than one element distinguishes the universal topology from $\TT$, any topological statement about it can easily be reformulated as a statement about $\TT$ and hence be expressed in our theory. Having said this, we can interpret the third axiom as stating that the universe is a $T_1$ space.

Alternatively one can understand the axioms without referring to collections outside the theory's scope as follows: Every set $a$ carries a topology $\square a$, and a union of two sets is a set again. Then the $T_1$ axiom says that all sets are $T_1$ spaces (and that all singletons are sets) and the fourth says that every set's hyperspace exists.

If $\VV$ is not a set, we cannot interpret the exponential axiom as saying that the universe's hyperspace exists! Since $\square a = \square a \cap \lozenge a$, it implies the power set axiom, but it does not imply the sethood of $\lozenge a$ for every set $a$.

A very handy implication of the 2nd topology axiom and the exponential axiom is that for all sets $b$, $c$ and every class $A$,
\[\{x{\in}c \mid A \subseteq x \subseteq b\} \quad = \quad \begin{cases}
c \; \cap \; \bigcap_{y{\in}A} (\square b \cap \lozenge \{y\}) & \text{ if $A\neq \emptyset$.}\\
(c \; \cap \; \square b) \cup (c\cap \Atom) & \text{ if $A=\emptyset$.}
\end{cases}\]
is closed, given that $c\cap \Atom$ is closed or $A$ is nonempty.

An important consequence of the $T_1$ axiom is that for each natural number\footnote{Until we have defined them in essential set theory, we consider natural numbers to be metamathematical objects.} $n$, all classes with at most $n$ elements are discrete sets. In particular, pairs are sets and we can define ordered pairs as Kuratowski pairs $\fTup{x}{y} = \{\{x\},\{x,y\}\}$. We adopt the convention that the $n{+}1$-tuple $\fTupL{x_1,\ldots,x_{n+1}}$ is $\fTup{\fTupL{x_1,\ldots,x_n}}{x_{n+1}}$ and that relations and functions are classes of ordered pairs. With these definitions, all functional formulas $\phi^C$ on sets correspond to actual functions, although these might be proper classes. We denote the $\in$-relation for sets by $\mathbf{E} = \{\fTup{x}{y} \mid x{\in} y\}$, and the equality relation by $\Delta =  \{\fTup{x}{y} \mid x{=} y\}$. Also, we write $\Delta_A$ for the equality $\Delta \cap A^2$ on a class $A$.

We have not yet made any stronger assumption than $T_1$ about the separation properties of sets. However, many desirable set-theoretic properties, particularly with respect to Cartesian products, apply only to Hausdorff sets, that is, sets whose natural topology is $T_2$.

We denote by $\square_{<n} A$ the class of all $b\subseteq A$ with less than $n$ elements. Given $t_1,\ldots,t_m\in\{1,\ldots,n\}$. We define:
\[F_{n,t_1,\ldots,t_m}:\VV^n \rightarrow \VV^m, F_{n,t_1,\ldots,t_m}(x_1,\ldots,x_n) = \fTupL{x_{t_1},\ldots,x_{t_m}}\]
With the corresponding choice of $t_1,\ldots, t_m$, all projections and permutations can be expressed in this way.

For a set $a$, let $a'$ be its \emphind{Cantor-Bendixson derivative}, the set of all its accumulation points, and let $a_I = a\setminus a'$ be the class of all its isolated points.

\begin{prp}[$\Th{ES}$]\label{ThmHausdorffSets}
Let $a$ and $b$ be Hausdorff sets and $a_1,\ldots, a_n \subseteq a$.
\begin{myenumerate}
\item\label{ThmHauSetPow} $\square_{<n} a$ is a Hausdorff set.
\item\label{ThmHauSetProd} The Cartesian product $a_1\times \ldots\times a_n$ is a Hausdorff set, too, and its universal topology is at least as fine as the product topology.
\item\label{ThmHauSetFun} Every continuous function $F: a_1\rightarrow a_2$ is a set.
\item\label{ThmHauSetProj} For all $t_1,\ldots,t_m\in\{1,\ldots,n\}$, the function
\[F_{n,t_1,\ldots,t_m} \upharpoonright a^n : a^n \rightarrow a^m\]
is a Hausdorff set. It is even closed with respect to the product topology of $a^{n+m}$.
\item\label{ThmHauSetInfProd} For each $x\in a_I$, let $b_x\subseteq b$. Then for every map $F:a_I \rightarrow b$, the class $F \cup (a'{\times}b)$ is $\TT$-closed and we can define the product as follows:
\[\prod_{x\in a_I} b_x \quad = \quad \left\{F \cup (a'{\times}b) \;\mid\; F:a_I \rightarrow \VV,\; \Fa{x} F(x)\in b_x \right\}\]
It is $\TT$-closed and its natural topology is at least as fine as its product topology.
\end{myenumerate}
\end{prp}

\begin{proof}
\eqref{ThmHauSetPow}: To show that it is a set it suffices to prove that it is a closed subset of the set $\square a$, so assume $b\in \square a \setminus \square_{<n} a$. Then there exist distinct $x_1,\ldots, x_n \in b$, which by the Hausdorff axiom can be separated by disjoint relatively open $U_1,\ldots, U_n \subseteq a$. Then $\lozenge U_1 \cap \ldots \cap \lozenge U_n \cap \square a$ is a relatively open neighborhood of $b$ disjoint from $\square_{<n} a$.

Now let $b,c\in \square_{<n} a$ be distinct sets. Wlog assume that there is a point $x\in b\setminus c$. Since $c$ is finite and $a$ satisfies the Hausdorff axiom, there is a relatively open superset $U$ of $c$ and a relatively open $V\ni x$, such that $U\cap V = \emptyset$. Now $\lozenge V \cap \square_{<n} a$ is a neighborhood of $b$ and $\square U \cap \square_{<n} a$ is a neighborhood of $c$ in $\square_{<n} a$, and they are disjoint. Hence $\square_{<n} a$ is Hausdorff.

\eqref{ThmHauSetProd}: It suffices to prove that $a\times a$ is a set and carries at least the product topology, because then it follows inductively that this is also true for $a^n$ with $n\geq 2$. And from this in turn it follows that $a_1 \times \ldots \times a_n$ is closed in $a^n$ and carries the subset topology, which implies the claim.

Since $a^2$ contains exactly the sets of the form $\{\{x\},\{x,y\}\}$ with $x,y\in a$, it is a subclass of the set $s = \square_{\leq 2}\square_{\leq 2} a \:\cap\: \lozenge\square_{\leq 1} a$ and we only have to prove that it is closed in $s$. So let $c \in s \setminus a^2$. Then $c=\{\{x\},\{y,z\}\}$ with $x\notin \{y,z\}$ and $x,y,z\in a$. Since $a$ is Hausdorff, there are disjoint $U \ni x$ and $V \ni y,z$ which are relatively open in $a$. Then $s\cap \lozenge \square_{\leq 1} U \cap \lozenge \square_{\leq 2} V$ is relatively open in $s$, and is a neighborhood of $c$ disjoint from $a^2$.

It remains to prove the claim about the product topology, that is, that for every subset $b\subseteq a$, $b\times a$ and $a\times b$ are closed, too. The first one is easy, because $b\times a = a^2 \cap \lozenge \square_{\leq 1} b$. Similarly, $(b\times a) \cup (a\times b) = a^2 \cap \lozenge \lozenge b$, so in order to show that $a\times b$ is closed, let $c\in (b\times a) \cup (a\times b) \setminus (a\times b)$, that is, $c=\{\{x\},\{x,y\}\}$ with $y\notin b$ and $x\in b$. Since $a$ is Hausdorff, there are relatively open disjoint subsets $U\ni x$ and $V\ni y$ of $a$. Then $s \cap \lozenge \square_{\leq 1} U \cap \lozenge \lozenge (V\setminus b)$ is a relatively open neighborhood of $c$ disjoint from $a\times b$.

\eqref{ThmHauSetFun}: Let $F:a_1 \rightarrow a_2$ be continuous and $\fTup{x}{y} \in a_1{\times} a_2 \setminus F$, that is, $F(x)\neq y$. Then $F(x)$ and $y$ can be separated by relatively open subsets $U\ni F(x)$ and $V\ni y$ of $a_2$, and since $F$ is continuous, $F^{-1}[U]$ is relatively open in $a_1$. $F^{-1}[U] \times V$ is a neighborhood of $\fTup{x}{y}$ and disjoint from $F$. This concludes the proof that $F$ is relatively closed in $a_1{\times} a_2$ and hence a set.

\eqref{ThmHauSetProj}: Let $F=F_{n,t_1,\ldots,t_m}$. Then $F\subseteq a^n \times a^m \in \VV$, so we only have to find for every
\[b = \fTup{\fTupL{x_1,\ldots,x_n}}{\fTupL{y_1,\ldots,y_m}} \text{, such that $x_{t_k} \neq y_k$ for some $k$,}\]
a neighborhood disjoint from $F$. By the Hausdorff property, there are disjoint relatively open $U\ni x_{t_k}$ and $V \ni y_k$. Then
\[\left(a^{t_k-1} \times U \times a^{n-t_k}\right) \times \left(a^{k-1}\times V \times a^{m-k}\right)\]
is such a neighborhood.

\eqref{ThmHauSetInfProd}: Firstly,
\[F\cup (a'{\times}b) \quad = \quad \bigcap_{x\in a_I} \left(\{\fTup{x}{F(x)}\} \;\cup\; ((a\setminus\{x\})\times b)\right)\]
is a set for any such function $F$.

Secondly, the claim about the product topology follows as soon as we have demonstrated the product to be $\TT$-closed, because the product topology is generated by classes of the form $\prod_{x\in a_I} c_x$, where $c_x\subseteq b_x$ and only finitely many $c_x$ differ from $b_x$.

Since $a\times b$ is $\TT$-closed and the product $P= \prod_{x\in a_I} b_x$ is a subset of $\square(a\times b)$, it suffices to show that $P$ is relatively closed in $\square(a\times b)$, so let $r\in \square(a\times b) \setminus P$. There are four cases:
\begin{myitemize}
\item The domain of $r$ is not $a$. Then there is an $x\in a$ such that $x\notin \dom (r)$. In that case, $\square(a\times b) \cap \lozenge (\{x\}\times b)$ is a closed superset of $P$ omitting $r$.
\item $a'\times b \nsubseteq r$. Then some $\fTup{x}{y}\in a'\times b$ is missing and $\square(a\times b) \cap \lozenge \{\fTup{x}{y}\}$ is a corresponding superset of $P$.
\item $r\upharpoonright a_I$ is not a function. Then there is an $x\in a_I$, such that there exist distinct $\fTup{x}{y_0},\fTup{x}{y_1} \in r$. Since $b$ is Hausdorff, there are closed $u_0,u_1\subseteq b$, such that $u_0\cup u_1=b$, $y_0\notin u_0$ and $y_1\notin u_1$. Then $P$ is a subclass of
\[\square\left( \{x\}\times u_0 \; \cup \; (a{\setminus} \{x\})\times b\right) \quad \cup \quad \square\left(\{x\}\times u_1 \; \cup \; (a{\setminus} \{x\})\times b\right)\text{,}\]
which does not contain $r$.
\item $F=r\upharpoonright a_I$ is a function, but $F(x) \notin b_x$ for some $x\in a_I$. Then
\[\square \left(\{x\}\times b_x \; \cup \; (a{\setminus} \{x\})\times b\right)\]
is a closed superclass of $P$ omitting $r$.
\end{myitemize}
Thus for every $r\in \square(a\times b) \setminus P$, there is a closed superclass of $P$ which does not contain $r$. Therefore $P$ is closed.
\end{proof}

The additivity axiom states that the universe is $\Discrete$-additive, that is, that the union of a discrete set's image is $\TT$-closed. In other words: For every function $F$ whose domain is a discrete set, the union of the range $\bigcup \rng(F)$ is a set or empty. Had we opted against proper classes, the additivity axiom therefore could have been expressed as an axiom scheme.

Even without a choice principle, we could equivalently have used injective functions into discrete sets instead of surjective functions defined on discrete sets: Point \eqref{ThmDTopDTop} in the following proposition is exactly the additivity axiom.

\begin{prp}\label{ThmDTop}
In $\Th{ES}$ without the additivity axiom, the following are equivalent:
\begin{myenumerate}
\item\label{ThmDTopUnionsImages} Images of discrete sets are sets, and unions of discrete sets are $\TT$-closed.
\item\label{ThmDTopDTop} If $d$ is discrete and $F:d\rightarrow A$ surjective, then $\bigcup A$ is $\TT$-closed.
\item\label{ThmDTopWeakDTop} If $d$ is discrete and $F:A\hookrightarrow d$ injective, then $\bigcup A$ is $\TT$-closed.
\end{myenumerate}
\end{prp}

\begin{proof}
\eqref{ThmDTopUnionsImages} $\Rightarrow$ \eqref{ThmDTopDTop}: If images of discrete sets are sets, then they are discrete, too, because all their subsets are images of subsets of a discrete set. Thus $F[d]$ is discrete, and therefore its union $\bigcup F[d]$ is closed.

\eqref{ThmDTopDTop} $\Rightarrow$ \eqref{ThmDTopUnionsImages}:
If $d$ is discrete and $F$ is a function, consider the function $G : \dom(F) \rightarrow \VV$ defined by $G(x) = \{F(x)\}$. Then $F[d]=\bigcup G[d]\in\VV$. Applying \eqref{ThmDTopDTop} to the identity proves that $\bigcup d = \bigcup \mathrm{id}[d]$ is closed.

\eqref{ThmDTopDTop} $\Rightarrow$ \eqref{ThmDTopWeakDTop}: If $F:A\hookrightarrow d$ is an injection, then $F^{-1} : F[A] \rightarrow A$ is a surjection from the discrete subset $F[A]\subseteq d$ onto $A$, so $\bigcup A$ is closed.

\eqref{ThmDTopWeakDTop} $\Rightarrow$ \eqref{ThmDTopDTop}: First we show that $\square d$ is discrete. We have to show that any given $a\in\square d$ is not an accumulation point, i.e. that $\square d \setminus \{a\}$ is closed. Since $a$ is a discrete set, every $d\setminus \{b\}$ for $b\in a$ is closed, as well as $d\setminus a$. But
\[\square d \setminus \{a\} = \square d \cap \left(\lozenge(d\setminus a) \cup \bigcup_{b\in a} \square (d\setminus \{b\})\right)\]
and this union can be seen to be closed by applying \eqref{ThmDTopWeakDTop} to the map
\[F:\{\square (d\setminus \{b\}) \mid b{\in} a\}\hookrightarrow a,\quad  F(\square (d\setminus \{b\})) = b\text{.}\]
Now we can prove \eqref{ThmDTopDTop}:

Let $G: d \rightarrow \VV$. Then $F: G[d] \rightarrow \square d, F(x) = G^{-1}[\{x\}]$ is an injective function from $G[d]$ to the discrete set $\square d$. Therefore, $\bigcup G[d]\in\VV$.
\end{proof}

\begin{prp}[$\Th{ES}$]
$\square d$ is discrete for every discrete set $d$. Every $\Discrete$-small nonempty class is a discrete set and every nonempty union of $\Discrete$-few discrete sets is a discrete set.
\end{prp}

\begin{proof}
The first claim has already been shown in the proof of Proposition \ref{ThmDTop}.

Let $A$ be $\Discrete$-small and $B\subseteq A$. Then $B$ and $\wt{B} = \{\{b\} \mid b {\in} B\}$ are also $\Discrete$-small. Therefore $\bigcup \wt{B} = B$ is $\TT$-closed by the additivity axiom.

Finally, let $A$ be $\Discrete$-small and let every $a\in A$ be a discrete set. We have to show that every nonempty $B \subseteq \bigcup A$ is a set. But if $A$ is $\Discrete$-small, the class $C$ of all nonempty sets of the form $B\cap a$ with $a\in A$ also is. Since $B\neq \emptyset$ and every $a\in A$ is discrete, the union of $C$ is in fact $B$.
\end{proof}

\section{Ordinal Numbers}

We do not assume that the empty class is a set, so there may be no well-founded sets at all, yet of course we want to define the natural numbers and later we will even be looking for an interpretation of a well-founded theory. To this end we need suitable variants of the concepts of well-foundedness and von Neumann ordinal numbers.

Our starting point is finding a substitute for the empty set: A class or atom $0$ is called a \emphind{zero} if no element of $0$ is a superset of $0$. Zeros exist in $\VV$: By the nontriviality axiom, there are distinct $x,y \in \VV$, so we can set $0=\{\{x\},\{y\}\}$. But in many interesting cases, there even is a definable zero: Let us set $0=\emptyset$ if $\emptyset\in \VV$, and if $\emptyset\notin \VV$ but $\VV\in\VV$, we set $0=\{\{\VV\}\}$ (its element $\{\VV\}$ is not a superset of $0$, because by the nontriviality axiom $\VV$ is not a singleton). Note that all these examples are sets with at most two elements.

Given a fixed zero $0$, we make the following definitions:
\begin{eqnarray*}
A^\oplus & = & A \setminus 0\\
A\in_0 B & \text{ if } & A \in B^\oplus \quad\text{ and }\quad 0 \subseteq B\text{.}\\
A \text{ is $0$-\emph{transitive}\index{transitive class}} & \text{ if } & c \in_0 A \quad \text{ for all } \quad c \in_0 b \in_0 A\text{.}\\
\text{A $0$-transitive } a \text{ is $0$-\emph{pristine}\index{pristine class}} & \text{ if } & 0\subseteq c \notin \Atom \quad\text{ for all }\quad c \in_0 a \cup \{a\}\text{.}\\
\alpha \text{ is a $0$-\emphind{ordinal number}} & \text{ if } & \alpha \text{ is $0$-transitive, $0$-pristine and}\\
&& \alpha^\oplus \text{ is strictly well-ordered by $\in_0$,}
\end{eqnarray*}
where by a (strict) \emphind{well-order} we mean a (strict) linear order such that each nonempty sub\emph{set} has a minimal element. A (strict) order with the property that every sub\emph{class} has a minimal element is called a (strict) \emph{strong well-order}\index{strong well-order}, and we will see shortly that in fact such $\alpha^\oplus$ are strictly strongly well-ordered.

We denote the class of $0$-ordinals by $\On_0$ and the $0$-ordinals themselves by lowercase greek letters. If $\alpha$ and $\beta$ are $0$-ordinals, we also write $\alpha\leq_0 \beta$ for $\alpha\subseteq \beta$. A $0$-ordinal $\alpha\neq 0$ is a $0$-\emphind{limit ordinal} if it is not the immediate $\leq_0$-successor of another $0$-ordinal, and it is a $0$-\emph{cardinal number} if there is no surjective map from $\beta^\oplus$ onto $\alpha^\oplus$ for any $\beta<_0 \alpha$. If there is a least $0$-limit ordinal distinct from $0$ itself, we call it $\omega_0$, otherwise we define $\omega_0 = \On_0$. Its predecessors $n\in_0 \omega_0$ are the $0$-\emph{natural numbers}. Obviously $0$ is the least $0$-ordinal, if $0\in \VV$.

For the remainder of this section, let us assume that our $0$ is an atom or a finite set. Unless there is danger of confusion (as in the case of $\in_0$), we omit the prefix and index $0$.

\begin{prp}[$\Th{ES}$]\label{LemOn}
Let $\alpha \in \On$.
\begin{myenumerate}
\item\label{LemOnDiscreteSet} $\alpha\notin\alpha$, $\alpha$ is discrete and $\alpha = 0 \; \cup \; \{ \beta \in \On \mid \beta\in_0 \alpha \}$.
\item\label{LemOnOrdered} $\On$ is strictly strongly well-ordered by $\in_0$ and $<$, and these orders coincide.
\item\label{LemOnSucc} $\alpha\cup\{\alpha\}$ is the unique immediate successor of $\alpha$.
\item\label{LemOnSup} If $A$ is a nonempty class of ordinals and $\bigcup A \in \VV$, then $\bigcup A$ is an ordinal and the least upper bound of $A$.
\item\label{LemOnProperClass} $\bigcup \On = \On \cup 0 \notin \VV$
\end{myenumerate}
\end{prp}

\begin{proof}
\eqref{LemOnDiscreteSet}:
Since $0\subseteq a$, the equality follows if we can prove that every $x\in_0 \alpha$ is an ordinal. Firstly, let $c\in_0 b \in_0 x$. Then $b\in_0 \alpha$ and $c\in_0 \alpha$ by transitiviy of $\alpha$. Since $\alpha^\oplus$ is strictly linearly ordered by $\in_0$, it follows that $c\in_0 x$, proving that $x$ is transitive. Again by the transitivity of $\alpha$, we see that $x\subseteq \alpha$, and as a subset of a well-ordered set, $x^\oplus$ is well-ordered itself. Also, every $c\in_0 x\cup \{x\}$ is an element of $\alpha^\oplus$ and therefore a superset of $0$ not in $\Atom$, so $x$ is pristine.

Since $\alpha$ is a superset of $0$, $\alpha\notin 0$. Thus if $\alpha$ were an element of $\alpha$, it would be in $\alpha^\oplus$. But $\alpha\in_0 \alpha$ contradicts the condition that the elements of $\alpha^\oplus$ are strictly well-ordered.

Because $0$ is a discrete set and $\alpha^\oplus = \{x \in \alpha \mid 0 \subseteq x\subseteq \alpha\}$ is closed, it suffices to show that $\alpha^\oplus$ is discrete. So let $\gamma \in_0 \alpha$. Since the elements of $\alpha^\oplus$ are strictly linearly ordered, every $\delta\in \alpha^\oplus \setminus \{\gamma\}$ is either a predecessor or a successor of $\gamma$. Hence
\[\alpha ^\oplus \setminus \{\gamma\}  \quad = \quad \gamma^\oplus \; \cup \; \{x \in \alpha^\oplus \mid \{\gamma\} \subseteq x \subseteq \alpha\}\]
is closed.

\eqref{LemOnOrdered}: If $\alpha\in_0 \beta$, then by transitivity of $\beta$, $\alpha$ is a subset of $\beta$ and because $\alpha\notin_0\alpha$, it is a proper one. For the converse assume $\alpha<\beta$, that is, $\alpha\subset \beta$. $\beta^\oplus$ is discrete and well-ordered, so the nonempty subset $\beta\setminus\alpha$ contains a minimal element $\delta$, which by \eqref{LemOnDiscreteSet} is an ordinal number. For all $\gamma\in_0\delta$, it follows from the minimality of $\delta$ that $\gamma\in_0\alpha$. Now let $\gamma\in_0\alpha$. Then $\gamma\in_0\beta$ and since $\beta$ is linearly ordered, $\gamma$ is comparable with $\delta$. But if $\delta\in_0\gamma$, then $\delta\in_0\alpha$ by transitivity, which is false. Hence $\gamma\in_0\delta$. We have shown that $\delta$ and $\alpha$ have the same predecessors, so by \eqref{LemOnDiscreteSet}, they are equal. Thus $\alpha=\delta\in_0\beta$ and so the orders $\in_0$ and $<$ coincide on the ordinals.

Next we show that ordinals $\alpha,\beta\in \On$ are always subsets of each other and hence $\On$ is linearly ordered, so assume they are not. Let $\alpha_0$ be minimal in $\alpha\setminus \beta$ and $\beta_0$ in $\beta\setminus \alpha$. Now all predecessors of $\alpha_0$ must be in $\alpha\cap\beta$. And since $\alpha$ and $\beta$ are transitive, $\alpha\cap\beta$ is an initial segment and therefore every element of $\alpha\cap\beta$ is also in $\alpha_0$. The same argument applied to $\beta_0$ shows that $\alpha_0 = \alpha\cap\beta = \beta_0$, contradicting our assumption.

Finally, given a nonempty subclass $A\subseteq \On$, let $\alpha\in A$ be arbitrary. Then either $\alpha$ has no predecessor in $A$ and thus is minimal itself, or $\alpha\cap A$ is nonempty and has a minimal element $\delta$, because $\alpha^\oplus$ is well-ordered and discrete and $\alpha\cap A \subseteq \alpha^\oplus$. For every $\gamma\in A \setminus \alpha$, we then have $\delta<\alpha\leq\gamma$. Hence $\delta$ is in fact minimal in $A$, concluding the proof that $\On$ is strongly well-ordered.

\eqref{LemOnSucc}:
First we verify that $\beta = \alpha\cup\{\alpha\}$ is an ordinal. Since $\alpha$ is transitive, $\beta$ also is. Since $\alpha$ is pristine and $0\subseteq \beta \notin \Atom$, $\beta$ is pristine itself. And $\beta^\oplus$ is a set of ordinal numbers, which by \eqref{LemOnOrdered} must be well-ordered.

From $\alpha\notin\alpha$ it follows that in fact $\beta\neq \alpha$ and thus $\beta>\alpha$. If $\gamma<\beta$, then $\gamma\in_0 \beta$, so either $\gamma \in_0 \alpha$ or $\gamma=\alpha$, which shows that $\beta$ is an immediate successor. Since the ordinals are linearly ordered, it is the only one.

\eqref{LemOnSup}: As a union of transitive, pristine, well-founded sets, $\bigcup A$ is transitive, pristine and well-founded itself. Since all its predecessors are ordinals, they are strictly well-ordered by \eqref{LemOnOrdered}, so it is an ordinal itself. For each $\beta\in A$, $\beta\subseteq \bigcup A$ and thus $\beta\leq \bigcup A$, so it is an upper bound of $A$. If $\beta<\bigcup A$, there is an element $\gamma\in A$ with $\beta<\gamma$, therefore it is the least upper bound.

\eqref{LemOnProperClass}: By \eqref{LemOnDiscreteSet}, every element $x$ of an ordinal is in $0\cup \On$. Conversely, $0$ is an ordinal and by \eqref{LemOnSucc}, every ordinal is an element of its successor. Therefore, $0\cup \On = \bigcup \On$. If $\bigcup \On$ were a set, so would $\bigcup \On \cup \{\bigcup\On\}$ be. But by \eqref{LemOnSup}, that would be an ordinal strictly greater than all elements of $\On$, which is a contradiction.
\end{proof}

These features of $\On$ are all quite desirable, and familiar from Zermelo-Fraenkel set theory. Just as in $\Th{ZF}$, $\On$ (or rather $\On \cup 0$) resembles an ordinal number itself, except that it is not a set. But in $\Th{ZF}$, $\On$ even has the properties of a regular limit cardinal -- a consequence of the replacement axiom. Also, our dependence on the choice of a specific set $0$ is rather irritating. This is where the additivity axiom comes in. In the context of ordinal numbers (and discrete sets in general), it is the appropriate analog to the replacement axiom.

By the usual argument, all strongly well-ordered classes whose initial segments are discrete sets are comparable with respect to their length: There is always a unique isomorphism from one of them to an initial segment of the other. In particular, for all finite zeros $0, \wt{0} \in \VV$, the well-orders of $\On_0$ and $\On_{\wt{0}}$ are comparable. But if $A\subseteq \On_0$ is an initial segment isomorphic to $\On_{\wt{0}}$, then in fact $A=\On_0$, because otherwise $A$ would be a discrete set and by the additivity axiom, $\On_{\wt{0}} \in \VV$, a contradiction. Hence $\On_0$ and $\On_{\wt{0}}$ are in fact isomorphic and the choice of $0$ is not relevant to our theory of ordinal numbers. Also, $\omega_0$ and $\omega_{\wt{0}}$ are equally long and we can define a class $A$ to be \emph{finite}\index{finite class} if there is a bijection from $n^\oplus$ to $A$ for some natural number $n$. Otherwise it is \emph{infinite}\index{infinite class}. It is easy to prove that this definition is equivalent to $A$ being the image of some $n^\oplus$ or embeddable into some $n^\oplus$. Also, it can be stated without quantifying over classes, because such a bijection is defined on a discrete set and therefore a discrete set itself.

Even if there is no limit ordinal, there might still be infinite sets -- they just cannot be discrete. So the proper axiom of infinity in the context of essential set theory is the existence of a limit ordinal number:
\begin{eqnarray*}
\text{\emph{Infinity}\index{axiom!infinity}} & \quad & \omega\in\VV
\end{eqnarray*}
We add the axiom of infinity to a theory by indexing it with the symbol $\infty$.

Using induction on ordinal numbers, one easily proves that for each $\alpha\in\On$, the least ordinal $\kappa\in \On$ such that there is a surjection from $\kappa^\oplus$ to $\alpha^\oplus$ is a cardinal, and there is a bijection from $\kappa^\oplus$ to $\alpha^\oplus$.

\begin{prp}[$\Th{ES}$]
$\On$ is a regular limit, that is:
\begin{myenumerate}
\item\label{ThmOnWCRegular} Every function $F:\alpha^\oplus\rightarrow\On$ is bounded.
\item\label{ThmOnWCLimit} The class of cardinal numbers is unbounded in $\On$.
\end{myenumerate}
\end{prp}

\begin{proof}
\eqref{ThmOnWCRegular}: By the additivity axiom, $\bigcup F[\alpha^\oplus]$ is a discrete set, so by Proposition \ref{LemOn}, it is an ordinal number and an upper bound of $F[\alpha^\oplus]$.

\eqref{ThmOnWCLimit}: Let us show that for each $\alpha$ there exists a cardinal $\nu>\alpha$. This goes by the usual argument: Every well-order $R \subseteq \alpha^\oplus \times \alpha^\oplus$ on a subset of $\alpha^\oplus$ is a subclass of the discrete set $\square \square \alpha^\oplus$, so it is a set itself and since $\alpha^\oplus$ is discrete, it is even a strong well-order. Recursively, isomorphisms from initial segments of $\alpha^\oplus$ with respect to $R$ to initial segments of $\On$ can be defined, and their union is a function from $\alpha^\oplus$ onto some $\beta^\oplus$. We call $\beta$ the \emphind{order type} of $R$. Now the class $A$ of all well-orders of $\alpha$ is a subclass of $\square \square \square \alpha$ and hence also a discrete set. Mapping every element of $A$ to its order type must therefore define a bounded map $F : A\rightarrow\On$. Let $\nu= \min (\On \setminus \bigcup F[A])$ be the least ordinal which is not an order type of any subset of $\alpha^\oplus$. We show that $\nu$ is a cardinal above $\alpha$. Firstly, $\alpha$ is the order type of a well-order of $\alpha^\oplus$, so $\nu>\alpha$. Secondly, assume that $g:\gamma^\oplus\rightarrow\nu^\oplus$ is surjective and $\gamma<\nu$. Then this defines a well-order on $\gamma^\oplus$ of order-type at least $\nu$, and since $\gamma$ is the order type of a well-order on some subset of $\alpha^\oplus$ by definition, $g$ would define a well-order on a subset of $\alpha^\oplus$ of order-type $\nu$, a contradiction.
\end{proof}

If $\VV\notin \VV$, the closure of $\On$ may well be all of $\VV$ and in particular does not have to be a set. But in the case $\VV \in\VV$, the fact that all $\lozenge a$ are sets determines the closure $\Omega$ of $0\cup \On = \bigcup \On$ much more precisely. Moreover, $\On$ then resembles a weakly compact cardinal, which will in fact turn out to be crucial for the consistency strength of the axiom $\VV \in\VV$.

\begin{prp}[$\Th{TS}$]\label{ThmVinV}
\begin{myenumerate}
\item \label{ThmVinVSequence} Every sequence $\fSeq{x_\alpha}{\alpha{\in}\On}$ of length $\On$ has an accumulation point.
\item \label{ThmVinVConvergence} Every monotonously $\subseteq$-decreasing sequence $\fSeq{x_\alpha}{\alpha{\in}\On}$ of nonempty sets converges to $\bigcap_{\alpha\in\On} x_\alpha$. And every monotonously $\subseteq$-increasing one to $\fCl\left(\bigcup_{\alpha\in\On} x_\alpha\right)$.
\item \label{ThmVinVOmega} $\Omega = 0 \cup \On \cup \{\Omega\}$
\item \label{ThmVinVPerfect} $P = \{x \mid 0 \cup \{\Omega\} \subseteq x \subseteq \Omega\}$ is a \emphind{perfect set}, that is, $P'=P \neq \emptyset$.
\item \label{ThmVinVWeaklyCompact} $\On$ has the tree property, that is: If
\[T \quad \subseteq \quad \{f:\alpha^\oplus\rightarrow \VV \mid \alpha \in \On\}\]
such that $T_\alpha = \{f{\upharpoonright} \alpha^\oplus \mid f{\in} T, \:\alpha^\oplus{\subseteq}\dom(f)\}$ is discrete and nonempty for each ordinal $\alpha>0$, then there is a $G:\On\rightarrow \VV$ such that:
\[G\upharpoonright \alpha^\oplus \in T_\alpha \quad\text{ for every }\quad \alpha\in\On\text{.}\]
\end{myenumerate}
\end{prp}

\begin{proof}
\eqref{ThmVinVSequence}:
Assume that there is no accumulation point. Then every point $y\in \VV$ has a neighborhood $U$ such that $\{\alpha \mid x_\alpha{\in} U\}$ is bounded in $\On$ and therefore discrete. Since the class $\{x_\alpha \mid x_\alpha{\in} U\}$ of members in $U$ is the image of $\{\alpha \mid x_\alpha{\in} U\}$, it is also a discrete set and does not have $y$ as its accumulation point. It follows that firstly, $\{\alpha \mid x_\alpha{=}y\}$ is discrete for each $y$, and secondly, the image $\{x_\alpha \mid \alpha{\in}\On\}$ of the sequence is also discrete. But $\On$ is the union of the sets $\{\alpha \mid x_\alpha{=}y\}$ for $y \in \{x_\alpha \mid \alpha{\in}\On\}$. Since $\Discrete$-small unions of discrete sets are discrete sets, this would imply that $\On$ is a discrete set, a contradiction.

\eqref{ThmVinVConvergence}: First let the sequence be decreasing. Then for every $y\in \bigcap_\alpha x_\alpha$, every member of the sequence lies in the closed set $\lozenge \{y\}$, so all its accumulation points do. Now let $y\notin \bigcap_\alpha x_\alpha$. Then there is a $\beta\in\On$ such that $y\notin x_\beta$, and hence from $x_\beta$ on, all members are in $\square x_\beta$, so all accumulation points are. Thus the only accumulation point is the intersection. (Note that the intersection therefore is nonempty because $\lozenge \VV$ is a closed set containing every member of the sequence.)

Now assume that the sequence is ascending and let $A$ be its union. If $y\in A$, then $y\in x_\beta$ for some $\beta\in\On$. Then all members from $x_\beta$ on are in $\lozenge \{y\}$, so each accumulation point also is. Thus all accumulation points are supersets of $A$. But all members of the sequence are in $\square \fCl(A)$, so each accumulation point is a subset of $\fCl(A)$, and therefore equal to $\fCl(A)$.

\eqref{ThmVinVOmega}:
It suffices to prove that $\Omega$ is the unique accumulation point of $\On$. Since $\On$ is the image of an increasing sequence, its accumulation point is indeed unique and is the closure of $\bigcup \On$ by \eqref{ThmVinVConvergence}. But $\fCl\left(\bigcup \On\right) = \fCl(0 \cup \On) = \Omega$.

\eqref{ThmVinVPerfect}:
$P$ is closed, and it is nonempty because $\Omega \in P$. Given $x\in P$, the sequences in $P$ given by
\[y_\alpha \;=\; x \setminus (\On \setminus \alpha^\oplus)\quad \text{ and } \quad z_\alpha \;=\; x \cup (\On \setminus \alpha^\oplus)\]
both converge to $x$ by \eqref{ThmVinVConvergence}. If $x \cap \On$ is unbounded, $x$ is not among the $y_\alpha$, otherwise it is not among the $z_\alpha$, so in any case, $x$ is the limit of a nontrivial sequence in $P$.

\eqref{ThmVinVWeaklyCompact}:
Since for every $\alpha\in \On$, $T_\alpha$ is nonempty, there is for every $\alpha$ an $f\in T$ with $\alpha^\oplus\subseteq \dom(f)$. Thus the map
\[T\rightarrow \On, \quad f \mapsto 0 \cup \dom(f)\]
is unbounded in $\On$ and therefore has a nondiscrete image. Hence $T$ is not discrete and has an accumulation point $g\in \VV$. We set $G= g\cap (\On \times \VV)$.

For each $\alpha\in\On$, the union $\bigcup_{\beta<\alpha} T_\beta$ is a discrete set, so $g$ is an accumulation point of the difference $T \setminus \bigcup_{\beta<\alpha} T_\beta$, which is the class of all those $f\in T$ whose domain is at least $\alpha^\oplus$. But every such $f$ is by definition the extension of some $h\in T_\alpha$. Thus this difference is the union of the classes $S_h = \{f\in T \mid h\subseteq f\}$ with $h\in T_\alpha$. Since $T_\alpha$ is discrete, $\fCl\left(\bigcup_{h\in T_\alpha} S_h\right) = \bigcup_{h\in T_\alpha} \fCl(S_h)$, so $g$ must be in the closure of some $S_h$. But $S_h$ is a subclass of the closed
\[\left\{x \mid h \subseteq x \subseteq h \cup \left(\Omega{\setminus}\alpha \times \VV\right)\right\}\text{,}\]
so $h\subseteq g \subseteq h \cup \left(\Omega{\setminus}\alpha \times \VV\right)$, too.

We have shown that for every $\alpha\in \On$, the set $g\cap \left(\alpha^\oplus \times \VV\right)$ is an element of $T_\alpha$. This implies that $G$ is a function defined on $\On$, and that $G\upharpoonright \alpha^\oplus = f \upharpoonright \alpha^\oplus$ for some $f\in T$, concluding the proof.
\end{proof}

In fact, we have just shown that \emph{every} accumulation point $g$ of $T$ gives rise to such a solution $G$. Hence firstly, $T = \fCl(T) \cap \square(\On\times\VV)$, and secondly, $G$ can always be described as the intersection of a set $g$ with $\On\times\VV$. In our formulation of the tree property, the two quantifications over classes could thus be replaced by quantifications over sets.

If $\Omega$ exists, the hierarchy of well-ordered sets extends well beyond the realm of ordinal numbers. By \emphind{linearly ordered set} we shall mean from now on a set together with a linear order $\leq$ such that the set's natural topology is at least as fine as the order topology, that is, such that all $\leq$-closed intervals are $\TT$-closed. And by \emphind{well-ordered set} we mean a linearly ordered set whose order is a well-order (or a strong well-order -- which in this case is equivalent). Then all well-ordered sets are comparable.

The significance of \eqref{ThmVinVPerfect} is that even if $\VV \in \VV$, the universe cannot be a well-ordered set, because well-ordered sets have no perfect subset. Thus whenever $a$ is a well-ordered set, there is a $p\notin a$ and the set $a\cup\{p\}$ can be well-ordered such that its order-type is the successor of the order-type of $a$. We use the usual notation for intervals in the context of linearly-ordered sets, and consider $\infty$ (resp. $-\infty$) as greater (resp. smaller) than all the elements of the set. We will also sloppily write $a+b$ and $a\cdot b$ for order-theoretic sums and products and say that an order-type \emph{exists} if there is a linearly ordered set with that order-type.

Every linearly ordered set $a$ is a Hausdorff set and since its order is closed with respect to the product topology, it is itself a set by Proposition \ref{ThmHausdorffSets}. Moreover, the class
\[\bigcap_{b{\subseteq}a \text{ initial segment}} \square b \;\cup\; \{c \mid b\subseteq c\subseteq a\}\]
of all its $\TT$-closed initial segments is itself a linearly ordered set in which $a$ can be embedded via $x\mapsto (-\infty,x]$. Thus we can limit our investigations to well-ordered sets whose order is given by $\subseteq$ and whose union exists, which makes things considerably easier:

\begin{lem}[$\Th{ES}$] If a class $A\subseteq \square a$ is linearly ordered by $\subseteq$, then $\fCl(A)$ is a linearly ordered set ordered by $\subseteq$. If $A$ is well-ordered, then so is $\fCl(A)$.
\end{lem}

\begin{proof}
First we prove that $\fCl(A)$ is still linearly ordered. Let $x,y\in \fCl(A)$ and assume that $x\nsubseteq y$. Every $z\in A$ is comparable to every other element of $A$, so $A$ is a subclass of the set $\square z \cup \{v \mid z\subseteq v \subseteq a\}$ and thus $\fCl(A)$ also is. Therefore both $x$ and $y$ are comparable to every element of $A$ and $A$ is a subclass of $\square x \cup \{v \mid x\subseteq v \subseteq a\}$. Since $y$ is not a superset of $x$, it must be in the closure of $A\cap \square x$ and thus a subset of $x$.

Since $\{v{\in}\fCl(A) \mid x\subseteq v \subseteq y\} = [x,y]$ is closed, $\fCl(A)$ in fact carries at least the order topology.

Now assume that $A$ is well-ordered and let $B\subseteq \fCl(A)$ be nonempty. Wlog let $B$ be a final segment. If $B$ has only one element, then that element is minimal, so assume it has at least two distinct elements. Since $A$ is dense, it must then intersect $B$ and $A\cap B$ must have a minimal element $x$. Assume that $x$ is not minimal in $B$. Then there is a $y\subset x$ in $B\setminus A$, and this $y$ must be minimal, because if there were a $z\subset y$ in $B$, then $(z,x)$ would be a nonempty open interval in $\fCl(A)$ disjoint from $A$.
\end{proof}

Thanks to this lemma, to prove that well-ordered sets of a certain length exist, it suffices to give a corresponding subclass of some $\square a$ well-ordered by $\subseteq$. As the next theorem shows, this enables us to do a great deal of well-order arithmetic in essential set theory.

\begin{prp}[$\Th{ES}$] If $a$ and $b$ are Hausdorff sets and $a_x \subseteq a$ is a well-ordered set for every $x\in b_I$, then $\sup_{x\in b_I} a_x$ exists. If in addition, $R$ is a well-order on $b_I$ (not necessarily a set), then $\sum_{x\in b_I} a_x$ exists. In particular, the order-type of $R$ exists, and binary sums and products of well-orders exist.
\end{prp}

\begin{proof}
Consider families $\fSeq{r_x}{x{\in}b_I}$ of initial segments $r_x\subseteq a_x$ with the following property: for all $x,y\in b_I$ such that $r_x\neq a_x$, the length of $r_y$ is the maximum of $r_x$ and $a_y$. Given such a family, the class
\[b' \times a \quad \cup \quad \bigcup_{x\in b_I} \{x\} \times r_x\]
is a set. And the class of all such sets is a subclass of $\square(b\times a)$ well-ordered by $\subseteq$ and at least as long as every $a_x$, because assigning to $y\in a_x$ the set
\[b' \times a \quad \cup \quad \bigcup_{z\in b_I} \{z\} \times r_z\text{,}\]
is an order-preserving map, where $r_z = a_z$ whenever $a_z$ is at most as long as $a_x$, and $r_z = (-\infty, \tilde{y}]$ such that $r_z$ is oder-isomorphic to $(-\infty, y]$ otherwise.

In the well-ordered case, consider for every $\fTup{x}{y} \in b_I \times a$ with $y\in a_x$ the set
\[b'\times a \quad \cup \quad \{x\} \times (-\infty,a_x] \quad \cup \quad (-\infty, x)_R \times a\text{.}\]
The class of these sets is again a subclass of $\square (b\times a)$ and well-ordered by $\subseteq$. Its order-type is the sum of the orders $a_x$.

Setting $a_x=1^\oplus$ for each $x$ yields a well-ordered set of the length of $R$. Using a two-point $b$ proves that binary sums exist. And if $b$ is a well-ordered set and $a_x = a$ for each $x\in b_I$, then $(b+1^\oplus)_I$ has at least the length of $b$ and $a\cdot b$ can be embedded in $\sum_{x\in (b+1^\oplus)_I} a_x$.
\end{proof}

\section{Pristine Sets and Inner Models}

Pristine sets are not only useful for obtaining ordinal numbers, but also provide a rich class of inner models of essential set theory and prove several relative consistency results. To this end, we need to generalize the notion of a pristine set, such that it also applies to non-transitive sets.

But first we give a general criterion for interpretations of essential set theory. The picture behind the following is this: The elements of the class $Z$ are to be ignored, so $Z$ is interpreted as the empty class. We do this to be able to interpret $\emptyset\in\VV$ even if the empty class is proper by choosing a nonempty set $Z\in\VV$. Everything that is to be interpreted as a class will be a superclass $X$ of $Z$, but only the elements of $X\setminus Z$ correspond to actual objects of the interpretation. In particular, $B \supseteq Z$ will be interpreted as the class of atoms and $W$ as the universe. So the extension of an element $x\in W\setminus B$ will be a set $X$ with $Z\subseteq X \subseteq W$, which we denote by $\Phi (x)$. Theorem \ref{ThmInnerModel} details the requirements these objects must meet to define an interpretation of $\Th{ES}$.

%\clearpage
\begin{thm}[$\Th{ES}$]\label{ThmInnerModel}
Let $\calK \subseteq \Discrete$ and $Z\subseteq B\subseteq W$ be classes and $\Phi : W\setminus B \rightarrow \VV$ injective. We use the following notation:
\begin{myitemize}
\item $X$ is an \emph{inner class} if it is not an atom and $Z\subseteq X \subseteq W$. In that case, let $X^\oplus = X \setminus Z$.
\item $S = W\setminus B^\oplus$ and $T = \Phi[S^\oplus]$.
\item $\overline{\Phi} = \Phi \cup \id_{B^\oplus} : W^\oplus \rightarrow \VV$
\end{myitemize}
Define an interpretation $\calI$ as follows:
\begin{eqnarray*}
X \text{ is in the domain of } \calI & \text{ if } & X \text{ is an inner class or } X \in B^\oplus\text{.}\\
X \in^\calI Y & \text{ if } & Y \text{ is an inner class and } X \in \overline{\Phi}[Y^\oplus]\text{.}\\
\Atom^\calI & = & B
\end{eqnarray*}
If the following conditions are satisfied, $\calI$ interprets essential set theory:
\begin{myenumerate}
\item\label{ThmInnerModelNontrivial} $W^\oplus$ has more than one element.
\item\label{ThmInnerModelInnerClasses} Every element of $T$ is an inner class, and no element of $B$ is an inner class.
\item\label{ThmInnerModelT1} $Z \cup \{x\} \in T$ for every $x\in W^\oplus$.
\item\label{ThmInnerModelTop2} Any intersection $\bigcap C$ of a nonempty $C \subseteq T$ is $Z$ or an element of $T$.
\item\label{ThmInnerModelTop3} $x\cup y \in T$ for all $x,y\in T$.
\item\label{ThmInnerModelDiscrete} If $x \in T$ and $x\setminus \{y\} \in T$ for all $y\in x^\oplus$, then $x^\oplus$ is $\calK$-small.
\item\label{ThmInnerModelAdditivity} Any union $\bigcup C$ of a nonempty $\calK$-small $C\subseteq T$ is an element of $T$.
\item\label{ThmInnerModelExp} For all $a,b\in T$, the class $Z \cup \left\{x{\in}S^\oplus \mid \Phi(x) {\subseteq} a, \Phi(x) {\cap} b {\neq}Z\right\}$ is $Z$ or in $T$.
\end{myenumerate}
The length of $\On^\calI$ is the least $\calK$-large ordinal $\kappa$, or $\On$ if no such $\kappa$ exists (for example in the case $\calK = \Discrete$). In particular, $(\omega\in\VV)^\calI$ iff $\omega$ is $\calK$-small.
\end{thm}

\begin{proof}
Let us first translate some $\calI$-interpretations of formulas:
\begin{myitemize}
\item $(X \notin \Atom)^\calI$ iff $X$ is an inner class, and $(X\in \Atom)^\calI$ iff $X\in B^\oplus$.
\item $(X \in \VV)^\calI$ iff $X\in \overline{\Phi}[W^\oplus]$, because $W$ is the union of all inner classes, so $\VV^\calI = W$.
\item If $(F:X_1\rightarrow X_2)^\calI$, then there is a function $G: \overline{\Phi}[X_1^\oplus] \rightarrow \overline{\Phi}[X_2^\oplus]$, defined by $G(Y_1) =  Y_2$ if $(F(Y_1)=Y_2)^\calI$, and $G$ is surjective resp. injective iff $(F \text{ is surjective})^\calI$ resp. $(F \text{ is injective})^\calI$.
\end{myitemize}
Now we verify the axioms of $\Th{ES}^\calI$:

\emph{Extensionality}: Assume $(X_1\neq X_2 \;\wedge\; X_1,X_2\notin \Atom)^\calI$. Then $X_1$ and $X_2$ are inner classes. But $X_1\neq X_2$ implies that there exists an element $y$ in $X_1 \setminus X_2 \subseteq W^\oplus$ or $X_2\setminus X_1 \subseteq W^\oplus$. $Y= \Phi(y)$ is either in $B^\oplus$ or an inner class by \eqref{ThmInnerModelInnerClasses}. Since $\Phi$ is injective, this means by definition that $(Y{\in}X_1 \wedge Y{\notin}X_2)^\calI$ or vice versa.

The \emph{atoms axiom} follows directly from our definition of $\in^\calI$, because no element of $B^\oplus$ is an inner class, and we enforced \emph{Nontriviality} by stating that $W^\oplus$ has more than one element.

\emph{Comprehension($\psi$):} If $Y = Z \cup \{ x{\in} W^\oplus \mid \psi^\calI(\Phi(x),\Vec{P}) \}$, then $Y$ witnesses the comprehension axiom for the formula $\psi=\phi^C$ with the parameters $\Vec{P}$, because $X \in^\calI Y$ iff
\[X\in \overline{\Phi}[Y^\oplus] = \{\Phi(x) \mid x{\in} W^\oplus \wedge \psi^\calI(\Phi(x),\Vec{P}) \}\text{,}\]
which translates to $X\in \overline{\Phi}[W^\oplus]$ and $\psi^\calI(X,\Vec{P})$.

\emph{$T_1$:} Let $(X\in\VV)^\calI$. Then $X=\overline{\Phi}(x)$ for some $x\in W^\oplus$. By \eqref{ThmInnerModelT1}, $Y = Z \cup \{x\} \in T = \Phi[S^\oplus]$, so in particular $(Y \in \VV)^\calI$. But $X$ is the unique element such that $X\in^\calI Y$, so $(Y = \{X\})^\calI$.

\emph{2nd Topology Axiom:} Assume $(D$ is a nonempty class of sets$)^\calI$, because if $(D$ contains an atom$)^\calI$, the intersection is empty in $\calI$ anyway. Then $D$ is an inner class and every $Y\in C = \overline{\Phi}[D^\oplus]$ is an inner class, which means $Y\in \Phi[S^\oplus]$. So $C \subseteq \Phi[S^\oplus]$ and $C\neq\emptyset$. We have $(X \in \bigcap D)^\calI$ iff $X\in^\calI Y$ for all $Y\in^\calI D$, that is:
\[X \in \bigcap_{Y\in C} \overline{\Phi}[Y^\oplus] = \overline{\Phi}\left[\left(\bigcap C\right)^\oplus\right]\text{,}\]
because $\overline{\Phi}$ is injective. Hence the inner class $\bigcap C$ equals $\left(\bigcap D\right)^\calI$, and by \eqref{ThmInnerModelTop2}, it is either in $T$ and therefore interpreted as a set, or it is $Z = \emptyset^\calI$.

\emph{Additivity}: A similar argument shows that $\bigcup C$ equals $\left(\bigcup D\right)^\calI$. If $(D$ is a discrete set$)^\calI$, then by \eqref{ThmInnerModelDiscrete}, $D^\oplus$ is $\calK$-small and therefore the union of $C = \overline{\Phi}[D^\oplus]$  is in $T$ by \eqref{ThmInnerModelAdditivity}.

\emph{3rd Topology Axiom:} Let $(X_1,X_2\in \TT)^\calI$. Then $X_1,X_2\in T$ and $X_1, X_2 \neq Z$. By \eqref{ThmInnerModelTop3}, $Y = X_1\cup X_2 \in T$, and $Y$ is interpreted as the union of $X_1$ and $X_2$.

The \emph{Exponential} axiom follows from \eqref{ThmInnerModelExp}, because
\[Y = Z \cup \left\{x{\in}S^\oplus \mid \Phi(x) {\subseteq} a, \Phi(x) {\cap} b {\neq} Z \right\}\]
equals $(\square a \cap \lozenge b)^\calI$. In fact, $X\in^\calI Y$ iff $X\in T$, $X\subseteq a$ and $X\cap b \neq Z$, and $X\subseteq a$ is equivalent to $(X\subseteq a)^\calI$, while $X\cap b \neq Z$ is equivalent to $(X\cap b \neq \emptyset)^\calI$.

The statement about the length of $\On^\calI$ holds true because the discrete sets are interpreted by the classes $X$ with $\calK$-small $X^\oplus$.
\end{proof}

All the conditions of the theorem only concern the image of $\Phi$ but not $\Phi$ itself, so given such a model one can obtain different models by permuting the images of $\Phi$. Also, if $\Phi[S^\oplus]$ is infinite and if $Z\in\VV$, one can toggle the truth of the statement $(\emptyset \in \VV)^\calI$ by including $Z$ in or removing $Z$ from $\Phi[S^\oplus]$.

\begin{prp}[$\Th{ES}$]\label{ThmHyperuniverseModel}
If $Z=\emptyset$, $T$ is a $\calK$-compact Hausdorff $\calK$-topology on $W$, $W$ has at least two elements, $B\subseteq W$ is open and does not contain any subsets of $W$, and $\Phi:W\setminus B \rightarrow \ExSpc_\calK (W,T)$ is a homeomorphism, then all conditions of Theorem \ref{ThmInnerModel} are met and therefore these objects define an interpretation of $\Th{ES}$. In addition, they interpret the statements $\VV\in \VV$ and that every set is $\Discrete$-compact Hausdorff.
\end{prp}

\begin{proof}
All conditions that we did not demand explicitly follow immediately from the fact that $W$ is a $\calK$-compact Hausdorff $\calK$-topological space and from the definition of the exponential $\calK$-topology.

$(\VV\in \VV)^\calI$ holds true, because $W\in \ExSpc_\calK(W,T)$. And since the $\calK$-small sets are exactly those interpreted as discrete, the $\calK$-compactness and Hausdorff property of $W$ implies that $(\VV$ is $\Discrete$-compact Hausdorff.$)^\calI$.
\end{proof}

Such a topological space $W$, together with a homeomorphism $\Phi$ to its hyperspace, is called a $\calK$-\emphind{hyperuniverse}. These structures have been extensively studied in \cite{FortiHonsellLenisa1995, FortiHonsell19962, Esser2003}. Here we will deal with a different class of models given by pristine sets.

Let $Z \subseteq B$ be such that no element of $B$ is a super\emph{set} of $Z$ (they are allowed to be atoms). Again, write $X\in_Z Y$ for:
\[X \in Y^\oplus \quad \text{ and } \quad Z \subseteq Y\text{.}\]
And $X$ is $Z$-\emphind{transitive} if $c\in_Z X$ whenever $c\in_Z b\in_Z X$. We say that $X$ is $Z$-$B$-\emph{pristine} if:
\begin{myitemize}
\item $X\in_Z B$ or:
\item $Z\subseteq X \notin \Atom$, and there is a $Z$-transitive set $b\supseteq X$, such that for every $c\in_Z b$ either $Z \subseteq c\notin \Atom$ or $c\in_Z B$.
\end{myitemize}
If $a$ has a $Z$-transitive superset $b$, then it has a least $Z$-transitive superset $\fTc(a) = \bigcap \{b{\supseteq}a \mid b \text{ $Z$-transitive}\}$, the $Z$-\emph{transitive closure} of $a$. Obviously a set is $Z$-transitive iff it equals its $Z$-transitive closure. Also, $a$ is $Z$-$B$-pristine iff $\fTc(a)$ exists and is $Z$-$B$-pristine. A set $a$ is $Z$-\emph{well-founded} iff for every $b\ni_Z a$, there exists an $\in_Z$-minimal $c\in_Z b$.

\begin{thm}[$\Th{ES}$]\label{ThmSpecialInnerModels}
Let $Z\in\VV$ and $B \supseteq Z$ such that no element of $B$ is a super\emph{set} of $Z$, and $B^\oplus$ is $\TT$-closed. Let $\Phi$ be the identity on $W\setminus B$ and $\calK = \Discrete$. The following classes $W_i^\oplus$ meet the requirements of Theorem \ref{ThmInnerModel} and therefore define interpretations $\calI_i$ of essential set theory:
\begin{myitemize}
\item the class $W_1^\oplus$ of all $Z$-$B$-pristine $x$
\item the class $W_2^\oplus$ of all $Z$-$B$-pristine $x$ with discrete $\fTc(x)^\oplus$
\item the class $W_3^\oplus$ of all $Z$-well-founded $Z$-$B$-pristine $x$ with discrete $\fTc(x)^\oplus$
\end{myitemize}
$Z$ is a member of all three classes and thus $(\emptyset\in \VV)^{\calI_i}$ holds true in all three cases. If $i\in\{2,3\}$, then $(\text{every set is discrete})^{\calI_i}$, and in the third case, $(\text{every set is $\emptyset$-well-founded})^{\calI_3}$.

If $\VV\in\VV$, then:
\begin{myenumerate}
\item\label{ThmSIModVinV} $(\VV \in \VV)^{\calI_1}$
\item\label{ThmSIModTree} $(\text{$\On$ has the tree property})^\calI_i$ for all $i$.
\item\label{ThmSIModCompr} If $B^\oplus$ is discrete, $\calI_3$ satisfies the strong comprehension principle.
\end{myenumerate}
\end{thm}

\begin{proof}
In this proof, we will omit the prefixes $Z$ and $B$: By ``pristine'' we always mean $Z$-$B$-pristine, ``transitive'' means $Z$-transitive and ``well-founded'' $Z$-well-founded.

Since $Z^\oplus$ is empty and $B$ is pristine and well-founded, $Z\in W_3^\oplus \subseteq W_2^\oplus \subseteq W_1^\oplus$.

Before we go through the requirements of Theorem \ref{ThmInnerModel}, let us prove that $x^\oplus$ is closed for every $x\in S^\oplus$:
\[x^\oplus \quad = \quad (x\cap B^\oplus) \; \cup \; (\{Z\}\cap x) \;\cup\; \{y\in x \mid Z \subseteq y\notin \Atom\}\]
Since $x$ is pristine, there is a transitive pristine $c\supseteq x$, and we can rewrite the class $\{y\in x \mid Z \subseteq y\notin \Atom\}$ as $\{y\in x \cap \square c\mid Z \subseteq y \subseteq c\}$, which is closed.

Condition \eqref{ThmInnerModelNontrivial} of Theorem \ref{ThmInnerModel} is satisfied because $Z$ and $Z \cup \{Z\}$ are distinct elements of $W_3^\oplus$.

\eqref{ThmInnerModelInnerClasses}: If $x\in B$, then $x$ is not a superset of $Z$ and therefore not an inner class. Now let $x\in S_1^\oplus$. We have to show that $x = \Phi(x)$ is an inner class. Since $x\notin B$ and $x$ is pristine, $x\notin \Atom$ and $Z\subseteq x$, so it only remains to prove that $y \in W_1^\oplus$ for every $y\in x^\oplus$. If $y \in_Z B$, $y$ is pristine. If $y\notin_Z B$, then $Z\subseteq y$. Since every transitive superset of $x$ is also a superset of $y$, $y$ is pristine in that case, too. If in addition, $\fTc(x)^\oplus$ is discrete, $y$ also has that property, by the same argument. And if $x$ is also well-founded, $y$ also is: For any $b\ni_Z y$, $b^\oplus\cup\{x\}$ has a $\in_Z$-minimal element; since $y\in_Z x$ and $y\in_Z b$, this cannot be $x$, so it must be in $b^\oplus$. This concludes the proof that $y\in W_i^\oplus$ whenever $x\in S_i^\oplus$.

\eqref{ThmInnerModelT1}: If $x\in W_1^\oplus$, then $Z \cup \{x\}$ is pristine, because if $x\in B^\oplus$, it is already transitive itself, and otherwise if $c$ is a transitive pristine superset of $x$, then $c\cup \{x\}$ is a transitive pristine superset of $Z\cup\{x\}$. If moreover $c^\oplus$ is discrete, then $c^\oplus\cup\{x\}$ also is, and if $x$ is well-founded, $Z\cup\{x\}$ also is.

\eqref{ThmInnerModelTop2}: Let $C\subseteq S_i^\oplus$ be nonempty. Then $\bigcap C \in S_i^\oplus$, too, because every subset of a pristine set which is a superset of $Z$ is pristine itself, every subset of a discrete set is discrete, and every subset of a well-founded set is well-founded.

\eqref{ThmInnerModelDiscrete}: Assume that for every $y \in x^\oplus$, we have $x\setminus \{y\}\in S^\oplus$. Then $(x\setminus \{y\})^\oplus = x^\oplus \setminus \{y\}$ is closed, and hence $x^\oplus$ is a discrete set.

\eqref{ThmInnerModelAdditivity} (and consequently \eqref{ThmInnerModelTop3}): Let $C\subseteq S_1^\oplus$ be a nonempty discrete set. Then $\bigcup C \in W_1\setminus B$, because if $c_b$ is a transitive pristine superset of $b$ for all $b\in C$, then $\bigcup c_b$ is such a superset of the union. If all the $c_b$ are discrete, their union also is, because they are only $\Discrete$-few. And if every element of $C$ is well-founded, $\bigcup C$ also is.

\eqref{ThmInnerModelExp}: $Y = Z \cup \left\{x{\in}S^\oplus \mid x {\subseteq} a, x {\cap} b {\neq}Z\right\}$ is pristine, because if $c$ is a transitive pristine superset of $a$, then $z = Z \cup \{Z\}\cup \{x{\in}\square c \mid Z {\subseteq} x \}$ is a transitive pristine superset of $Y$. And $Y$ is in fact a set, because $b^\oplus$ is closed, so $Y= Z \cup (z^\oplus \cap \square a \cap \lozenge b^\oplus)$ also is. If $c^\oplus$ is discrete, $\square c^\oplus$ is discrete, and so is $z^\oplus \setminus\{Z\} = \{y \cup Z \mid y\in \square c^\oplus\}$. And if $a$ is well-founded, any set of subsets of $a$ is well-founded, too.

The claims about discreteness and well-foundedness are immediate from the definitions.

Now let us prove the remaining claims under the assumption that $\VV \in \VV$:

\eqref{ThmSIModVinV}: $\VV\setminus\Atom$ is a set, namely $\lozenge \VV \cup \{\emptyset\}$ or $\lozenge \VV$, depending on whether $\emptyset\in\VV$. Let:
\begin{eqnarray*}
U_0 &=& \VV\\
U_{n+1} &=& B^\oplus \;\cup\; \{x\in\VV{\setminus}\Atom \mid Z \subseteq x \subseteq Z \cup U_n\}\\
U_\omega &=& \bigcap_{n\in\omega} U_n
\end{eqnarray*}
Then $U_\omega$ is a set. Since $W_1^\oplus \subseteq \VV$ and $W_1^\oplus \subseteq B^\oplus \cup \{x\in\VV{\setminus}\Atom \mid Z \subseteq x\subseteq Z \cup  W_1^\oplus\}$, it is a subset of $U_\omega$. It remains to show that $U_\omega \subseteq W_1^\oplus$, that is, that every element of $U_\omega$ is pristine, because then it follows that $W_1$ is a pristine set itself and hence $W_1\in_Z W_1$. In fact, it suffices to prove that $Z\cup U_\omega$ is a transitive pristine set, because then all $x\in_Z U_\omega$ will be pristine, too. So assume $y\in_Z x\in_Z Z \cup U_\omega$. If $x$ were in $B^\oplus$, then $y\notin_Z x$, so $x$ must be in $\VV\setminus\Atom$ and $Z \subseteq x \subseteq Z \cup U_n$ for all $n$. Thus $x\subseteq Z \cup U_\omega$, which implies that $y\in_Z U_\omega$.

\eqref{ThmSIModTree} follows from Proposition \ref{ThmVinV}.

\eqref{ThmSIModCompr}: It suffices to show that $W_3^\oplus$ does not contain any of its accumulation points, because that implies that every inner class corresponds to a set -- it's closure --, so that the weak comprehension principle allows us to quantify over all inner classes. Since $B^\oplus$ is discrete and
\[S_3^\oplus \setminus \{Z\} \quad = \quad W_3^\oplus \setminus  (B \cup \{Z\}) \quad\subseteq\quad \{x\in\VV{\setminus}\Atom \mid Z \subseteq x\}\quad\in\quad \VV\] (recall that no element of $B$ is a superset of $Z$), $B$ certainly contains no accumulation point of $W_3^\oplus$. So assume now that $x\in W_3^\oplus$ is an accumulation point. Since it is well-founded and $\fTc(x)^\oplus$ is a discrete set, $\fTc(x)^\oplus\cup \{x\}$ has an $\in_Z$-minimal $W_3^\oplus$-accumulation point $y$. Then $y\in S_3^\oplus$ and $y$ is also an accumulation point of $W_3^\oplus \setminus (B^\oplus \cup \{Z\})$. Since none of the $\Discrete$-few elements of $y^\oplus$ is an $W_3^\oplus$-accumulation point, $W_3^\oplus\setminus (B^\oplus \cup \{Z,y\})$ is a subclass of
\[\lozenge \fCl(W_3^\oplus \setminus y) \;\cup\; \bigcup_{z\in_Z y} \square \fCl(W_3^\oplus \setminus \{z\})\text{,}\]
which is closed and does not contain $y$, a contradiction.
\end{proof}

By the nontriviality axiom, there are distinct $x,y\in\VV$. If we set $Z = B = \{\{x\},\{y\}\}$, the requirements of Theorem \ref{ThmSpecialInnerModels} are satisfied, so $\calI_i$ interprets essential set theory with $\emptyset \in \VV$ in all three cases. Moreover, since $Z=B$, it interprets $\Atom=\emptyset$. So $\Atom=\emptyset\in\VV$ is consistent relative to $\Th{ES}$. In the case $i=3$, moreover, $($every set is $\emptyset$-well-founded and discrete$)^{\calI_3}$! And if in addition $\omega \in \VV$, then $\omega$ is $\Discrete$-small and thus $(\omega\in\VV)^{\calI_3}$ by Theorem \ref{ThmInnerModel}.

But if in $\Th{ES}$ every set is discrete and $\emptyset$-well-founded, the following statements are implied:
\begin{eqnarray*}
\text{\emph{Pair, Union, Power, Empty Set}} && \{a, b\}, \; \bigcup a, \; \mathfrak{P} (a), \; \emptyset \; \in \; \VV\\
\text{\emph{Replacement}} && \text{If $F$ is a function and $a \in \VV$, then $F [a] \in \VV$.}\\
\text{\emph{Foundation}} && \text{Every $x\in\TT$ has a member disjoint from itself.}
\end{eqnarray*}
And these are just the axioms of $\Th{ZF}$\footnote{With classes, of course. We avoid the name $\Th{NBG}$, because that is usually associated with a strong axiom of choice.}! Conversely, all the axioms of $\Th{ES}$ hold true in $\Th{ZF}$, so $\Th{ZF}$ could equivalently be axiomatized as follows\footnote{
We will soon introduce a choice principle for $\Th{ES}$, the uniformization axiom, which applies to all discrete sets. Since in $\Th{ZF}$ every set is discrete, that axiom is equivalent to the axiom of choice.}:
\begin{myitemize}
\item $\Th{ES_\infty}$
\item $\Atom = \emptyset \in \VV$
\item Every set is discrete and $\emptyset$-well-founded.
\end{myitemize}
If in addition $\VV\in\VV$, then $\calI_3$ even interprets the strong comprehension axiom and therefore Kelley-Morse set theory\footnote{The axiom of choice is not necessarily true in that interpretation, but even the existence of a global choice function does not add to the consistency strength, as was shown in \cite{Esser2004}.} with $\On$ having the tree property. Conversely, O. Esser showed in \cite{Esser1997} and \cite{Esser1999} that this theory is equiconsistent with $\Th{GPK}^+_\infty$, which in turn is an extension of topological set theory that will be introduced in the next section. In summary, we have the following results:

\begin{cor}
$\Th{ES}_\infty$ is equiconsistent with $\Th{ZF}$: The latter implies the former and the former interprets the latter.

$\Th{TS}_\infty$ and $\Th{GPK}^+_\infty$ both are mutually interpretable with:\\
Kelley-Morse set theory $\;+\;$ $\On$ has the tree property.
\end{cor}

$\calI_3$ is a particularly intuitive interpretation if $\emptyset, \VV \in \VV$, $\Atom=\emptyset$ and we set $Z=B=\emptyset$. Then every set is ($\emptyset$-$\emptyset$-)pristine and $\in_N$ is just $\in$. Also, $\VV \setminus \{\emptyset\} = \lozenge \VV \in \VV$, so $\emptyset$ is an isolated point. If a set $x$ contains only isolated points, it is discrete, and since $x= \bigcup_{y\in x} \{y\}$ and every $\{y\}$ is open, $x$ is a clopen set. Moreover, $x$ is itself an isolated point, because $\{x\}$ is open:
\[\{x\} \quad = \quad \square x \cup \bigcup_{y\in x} \lozenge \{y\}\]
Thus it follows that all ($\emptyset$-)well-founded sets are isolated. Define the cumulative hierarchy as usual:
\begin{eqnarray*}
U_0 & = & \emptyset\\
U_{\alpha+1} & = & \square U_\alpha \cup \{\emptyset\}\\
U_\lambda & = & \bigcup_{\alpha<\lambda} U_\alpha \; \text{ for limit ordinals } \lambda
\end{eqnarray*}
Since images of discrete sets in $\On$ are bounded and since every nonempty class of well-founded sets has an $\in$-minimal element, the union $\bigcup_{\alpha\in\On} U_\alpha$ is exactly the class of all well-founded sets, and in fact equals $W_3$.

\section{Positive Specification}

This section is a short digression from our study of essential set theory. Again starting from only the class axioms we introduce specification schemes for two classes of ``positive'' formulas as well as O. Esser's theory $\Th{GPK^+}$ (cf. \cite{Esser1997,Esser1999,Esser2000,Esser2004}), and then turn our attention to their relationship with topological set theory.

The idea of positive set theory is to weaken the inconsistent \emphind{naive comprehension scheme} -- that every class $\{x \mid \phi(x)\}$ is a set -- by permitting only \emph{bounded positive formula}s (BPF), which are defined recursively similarly to the set of all formulas, but omitting the negation step, thus avoiding the Russell paradox. This family of formulas can consistently be widened to include all \emph{generalized positive formula}s (GPF), which even allow universal quantification over classes. But to obtain more general results, we will investigate \emph{specification} schemes instead of comprehension schemes, which only state the existence of subclasses $\{x{\in}c \mid \phi(x)\}$ of sets $c$. If $\VV$ is a set, this restriction makes no difference.

We define recursively when a formula $\phi$ whose variables are among $X_1,X_2,\ldots$ and $Y_1, Y_2,\ldots$ (where these variables are all distinct) is a \emphind{generalized positive formula} (GPF) \emph{with parameters} $Y_1, Y_2,\ldots$:
\begin{myitemize}
\item The atomic formulas $X_i\in X_j$ and $X_i = X_j$ are GPF with parameters $Y_1, Y_2,\ldots$.
\item If $\phi$ and $\psi$ are GPF with parameters $Y_1, Y_2,\ldots$, then so are $\phi\wedge \psi$ and $\phi \vee \psi$.
\item If $i\neq j$ and $\phi$ is a GPF with parameters $Y_1, Y_2,\ldots$, then so are $\Fa{X_i{\in}X_j} \phi$ and $\Ex{X_i{\in}X_j} \phi$.
\item If $\phi$ is a GPF with parameters $Y_1, Y_2,\ldots$, then so is $\Fa{X_i {\in} Y_j} \phi$.
\end{myitemize}
A GPF with parameters $Y_1, Y_2,\ldots$ is a \emphind{bounded positive formula} (BPF) if it does not use any variable $Y_i$, that is, if it can be constructed without making use of the fourth rule. The \emphind{specification axiom} for the GPF $\phi(X_1, \ldots, X_m,Y_1,\ldots,Y_n)$ with parameters $Y_1,Y_2,\ldots$, whose free variables are among $X_1,\ldots, X_m$, is:
\begin{eqnarray*}
&\Class{x{\in}c}{\phi(x, b_2,\ldots, b_m, B_1,\ldots, B_n)} \text{ is $\TT$-closed}&\\
&\text{for all $c, b_2,\ldots, b_m \in \VV$ and all classes $B_1,\ldots, B_n$.}&
\end{eqnarray*}
\emphind{GPF specification} is the scheme consisting of the specification axioms for all GPF $\phi$, and \emphind{BPF specification} incorporates only those for BPF $\phi$. Note that we did not include the formula $x\in \Atom$ or any other formula involving the constant $\Atom$ in the definition, so $x\in\Atom$ is not a GPF.

The following theorem shows that BPF specification is in fact finitely axiomatizable, even without classes.\footnote{A similar axiomatization, but for positive \emph{comprehension}, is given by M. Forti and R. Hinnion in \cite{FortiHinnion1989}. On the other hand, no finite axiomatization exists for \emph{generalized} positive comprehension, as O. Esser has shown in \cite{Esser2004}.}

\begin{thm}\label{ThmPosCompAxioms}
Assume only the class axioms and that for all $a,b\in\VV$, the following are $\TT$-closed:
\[\bigcup a,\quad \{a, b\},\quad a{\times}b\]
Let $\Theta$ be the statement that for all sets $a,b\in\VV$, the following are $\TT$-closed:
\begin{eqnarray*}
&\Delta{\cap}a,\quad \mathbf{E}{\cap}a, \quad\Class{\fTup{x}{y}{\in}b}{\Fa{z{\in}y} \fTupL{x,y,z}{\in}a},&\\
&\Class{\fTupL{y, x, z}}{\fTupL{x, y, z}{\in}a},\quad \Class{\fTupL{z, x, y}}{\fTupL{x, y, z}{\in}a}&
\end{eqnarray*}
Then BPF specification is equivalent to $\Theta$. And GPF specification is equivalent to $\Theta$ and the second topology axiom.
\end{thm}

\begin{proof}
Ordered pairs can be built from unordered ones, and the equality $\fTup{x}{y} = z$ can be expressed as a BPF. Therefore the classes mentioned in $\Theta$ can all be defined by applying BPF specification to a given set or product of sets, so BPF specification implies $\Theta$.

GPF specification in addition implies the second topology axiom,
\[\Fa{B{\neq}\emptyset.} \quad \emptyset{=}\bigcap B \quad \vee \quad \bigcap B \in \VV\text{,}\]
because $\Fa{a{\in}B} x{\in}a$ is clearly a GPF with parameter $B$, and the intersection is a subclass of any $c\in B$.

To prove the converse, assume now that $\Theta$ holds. Since it is not yet clear what we can do with sets, we have to be pedantic with respect to Cartesian products. We define
\[A \times_2 B = \left\{\fTupL{a,b_1,b_2} \mid a{\in}A, \fTup{b_1}{b_2}{\in}B\right\}\text{,}\]
which is not the same as $A\times B$ for $B\subseteq \VV^2$, because $\fTupL{a,b_1,b_2} = \fTup{\fTup{a}{b_1}}{b_2}$, whereas the elements of $A\times B$ are of the form $\fTup{a}{\fTup{b_1}{b_2}}$. Yet we can construct this and several other set theoretic operations from $\Theta$:
\begin{eqnarray*}
a\times_2 b &=& \Class{\fTupL{z, x, y}}{\fTupL{x, y, z} \in b\times a}\\
a\cup b &=& \bigcup \{a,b\}\\
a\cap b &=& \bigcup \bigcup \{\{\{x\}\} \mid x \in a\cap b\} = \bigcup \bigcup(\Delta \cap (a{\times}b))\\
a \cap \VV^2 &=& a \:\cap\: \left(\bigcup \bigcup a\right)^2\\
\{\{x\} \mid \{x\} \in a\} &=& a \: \cap \: \bigcup\left(\Delta \cap \left(\bigcup a\right)^2\right)\\
\dom (a) &=& \bigcup \left\{\{x\} \mid \{x\} \in \bigcup (a \cap \VV^2)\right\}\\
a^{-1} &=& \dom\left(\{\fTupL{y, x, z} \mid \fTupL{x,y,z} \in  a{\times}\{a\}\}\right)
\end{eqnarray*}

We will prove by induction that for all GPF $\phi (X_1,\ldots,X_m,Y_1,\ldots,Y_n)$ with parameters $Y_1,\ldots, Y_n$ and free variables $X_1,\ldots,X_m$, and for all classes $B_1,\ldots, B_n$ and sets $a_1,\ldots, a_m$,
\[A^\phi_{a_1,\ldots,a_m} = \Class{\fTupL{x_1,\ldots,x_m} \in a_1{\times}\ldots{\times}a_m}{\phi(x_1,\ldots,x_m, B_1,\ldots,B_n)}\]
is $\TT$-closed. This will prove the specification axiom for $\phi$, because
\[\Class{x{\in}c}{\phi(x, b_2,\ldots, b_m, B_1,\ldots, B_n)} \quad = \quad \dom \left( \ldots \dom \left( A^\phi_{c,\{b_2\},\ldots,\{b_m\}} \right) \ldots\right)\text{,}\]
where the domain operation is applied $m-1$ times.

Each induction step will reduce the claim to a subformula or to a formula with fewer quantifiers. Let us assume wlog that no bound variable is among the $X_1,\ldots$ or $Y_1,\ldots$ and just always denote the bound variable in question by $Z$.

{\bf Case 1:} Assume $\phi$ is $\Fa{Z{\in} Y_i} \psi$. Then
\[A^\phi_{a_1,\ldots,a_m} = \bigcap_{x\in B_i} \dom \left(A^{\psi(Z / X_{m+1})}_{a_1,\ldots,a_m, \{x\}}\right)\text{,}\]
where $\psi(Z / X_{m+1})$ is the formula $\psi$, with each free occurrence of $Z$ substituted by $X_{m+1}$. This is the step which is only needed for GPF formulas. Since it is the only point in the proof where we make use of the closure axiom, we otherwise still obtain BPF specification as claimed in the theorem.

{\bf Case 2:} Assume $\phi$ is a bounded quantification. If $\phi$ is $\Ex{Z{\in} X_i} \psi$, then
\[A_{\phi,a_1,\ldots,a_m} = \dom \left(A^{\psi(Z/X_{m+1}) \;\wedge\; X_{m+1} {\in} X_i}_{a_1,\ldots,a_m, b}\right)\text{,}\]
where $b = \bigcup a_i$. If $\phi$ is $\Fa{Z {\in}X_i} \psi$, then
\[A^\phi_{a_1,\ldots,a_m} = \dom \left\{\fTupL{x,y} \in a_1{\times}{\ldots}{\times}a_m{\times}a_i \mid \Fa{z {\in} y} \fTupL{x,y,z} \in A^{\rho}_{a_1,\ldots,a_m, a_i, b} \right\}\text{,}\]
where again $b = \bigcup a_i$, and $\rho$ is the formula $\psi(Z/X_{m+2}) \;\wedge\;X_{m+1}{=}X_i$. The class defined here is of the form $\Class{\fTup{x}{y}{\in}b}{\Fa{z{\in}y} \fTupL{x,y,z}{\in}a}$ and therefore a set, by our assumption.

{\bf Case 3:} Assume $\phi$ is a conjunction or disjunction. If $\phi$ is $\psi \wedge \chi$ resp. $\psi \vee \chi$, then
\[A^\phi_{a_1,\ldots,a_m} = A^\psi_{a_1,\ldots,a_m} \cap A^\chi_{a_1,\ldots,a_m} \quad\text{ resp. }\quad A^\phi_{a_1,\ldots,a_m} = A^\psi_{a_1,\ldots,a_m} \cup A^\chi_{a_1,\ldots,a_m}\text{.}\]

{\bf Case 4:} Assume $\phi$ is atomic. If $X_m$ does not occur in $\phi$, then $A^\phi_{a_1,\ldots,a_m} = A^\phi_{a_1,\ldots,a_{m-1}} \times a_m$. If $\phi$ has more than one variable, but $X_{m-1}$ is not among them, then:
\[A^\phi_{a_1,\ldots,a_m} = \left\{ \fTupL{z, x_{m-1}, x_m} \mid \fTupL{z, x_m, x_{m-1}} \in A^{\phi(X_m/X_{m-1})}_{a_1,\ldots,a_{m-2},a_m} \times a_{m-1} \right\}\]
Applying these two facts recursively reduces the problem to the case where either $m=1$ or where $X_m$ and $X_{m-1}$ both occur in $\phi$:
\begin{eqnarray*}
A^{X_1 = X_1}_{a_1} & = & a_1\\
A^{X_1 \in X_1}_{a_1} & = & \dom\left(\mathbf{E} \cap a_1^2\right)\\
A^{X_{m-1} = X_m}_{a_1,\ldots,a_m} & = & a_1\times\ldots\times a_{m-2} \times_2 (\Delta \cap (a_{m-1}{\times}a_m))\\
A^{X_m = X_{m-1}}_{a_1,\ldots,a_m} & = & a_1\times\ldots\times a_{m-2} \times_2 (\Delta \cap (a_{m-1}{\times}a_m))\\
A^{X_{m-1} \in X_m}_{a_1,\ldots,a_m} & = & a_1\times\ldots\times a_{m-2} \times_2 (\mathbf{E} \cap (a_{m-1}{\times}a_m))\\
A^{X_m \in X_{m-1}} & = & a_1\times\ldots\times a_{m-2} \times_2 (\mathbf{E}^{-1} \cap (a_{m-1}{\times}a_m))
\end{eqnarray*}
\end{proof}

As we already indicated, the theory $\Th{GPK}^+$ uses GPF \emph{comprehension}, but if $\VV\in\VV$, specification entails comprehension. $\Th{GPK}^+$ can be axiomatized as follows:
\begin{myitemize}
\item $\VV\in\VV$
\item $\Atom = \emptyset \in \VV$
\item GPF specification
\end{myitemize}

\begin{prp}\label{LemPosTop}
$\Th{GPK}^+$ implies $\Th{TS}$ and that unions of sets are sets.
\end{prp}

\begin{proof}
If $B\subseteq \TT$, then $\bigcap B = \{x\mid \Fa{y{\in}B} x{\in}y\}$ is $\TT$-closed, and if $a,b\in \TT$, then $a\cup b = \{x \mid x{\in}a \vee x{\in}b\} \in \VV$, because these are defined by GPFs, proving the 2nd and 3rd topology axioms. $\{a\} = \{x \mid x{=}a\}$ and $x{=}a$ is bounded positive, so $T_1$ is also true.

$\square a \cap \lozenge b = \{c \mid \Ex{x{\in}b} x{=}x \wedge \Fa{x{\in}c} x{\in}a \wedge \Ex{x{\in}b} x{\in}c\}$ is defined by a positive formula as well, yielding the exponential axiom.

$\bigcup a = \{c \mid \Ex{x{\in}a} c{\in}x\}$ is also $\TT$-closed, for the same reason.

The formula $z=\{x,y\}$ can be expressed as $x{\in}z \wedge y{\in}z \wedge \Fa{w{\in}z} (w{=}x \vee w{=}y)$, so it is bounded positive. Using that, we see that ordered pairs, Cartesian products, domains and ranges can all be defined by GPFs. This allows us to prove the additivity axiom:

Let $a\in\TT$ be discrete and $F:a\rightarrow\VV$. We first show that $F\in \VV$: Firstly, $F\subseteq a\times \VV$ and $a\times \VV$ is $\TT$-closed. Secondly, if $\fTup{x}{y}\in (a\times \VV)\setminus F$, then $F(x)\neq y$, so $F$ is a subclass of the $\TT$-closed $(a{\setminus}\{x\} \times \VV) \cup \{\fTup{x}{F(x)}\}$, which does not contain $\fTup{x}{y}$. Thus $F$ is a set and hence $\bigcup \rng(F)$ is $\TT$-closed.
\end{proof}

\section{Regularity and Union}

After having seen that topological set theory is provable in $\Th{GPK}^+$, we now aim for a result in the other direction. To this end we assume in addition to $\Th{ES}$ the union axiom and that every set is a regular space:
\begin{eqnarray*}
\text{\emph{Union}\index{axiom!union}} & \quad & \bigcup a \text{ is $\TT$-closed for every $a\in\VV$.}\\
\text{\emph{$T_3$}\index{axiom!regularity}} & \quad & x{\in}a \;\wedge\; b{\in}\square a \quad\Rightarrow\quad \Ex{u,v.} \; u{\cup}v{=}a \;\wedge\; x{\notin}u \;\wedge\; b{\cap}v{=}\emptyset
\end{eqnarray*}
These two axioms elegantly connect the topological and set-theoretic properties of orders and products. Note that they, too, are theorems of $\Th{ZF}$, because every discrete set is regular and its union is a set.

Recall that we use the term \emphind{ordered set} only for sets with an order $\leq$, whose order-topology is at least as fine as their natural topology. By default, we consider the order itself to be the non-strict version.

\begin{prp}[$\Th{ES}+\text{Union}+T_3$]\label{ThmSep}
\begin{myenumerate}
\item\label{ThmSepDom} Domains and ranges of sets are sets.
\item\label{ThmSepMap} Every map in $\VV$ is continuous and closed with respect to the natural topology.
\item\label{ThmSepLin} A linear order $\leq$ on a set $a$ is a set iff its order topology is at most as fine as the natural topology of $a$.
\item\label{ThmSepFinProd} The product topology of $a^n$ is equal to the natural topology.
\item\label{ThmSepGPF} If $\Atom$ is closed, GPF specification holds.
\end{myenumerate}
\end{prp}

\begin{proof}
\eqref{ThmSepDom}: Let $a$ be a set. Then $c=\bigcup \bigcup a$ is a set, and in fact, $c=\dom(a) \cup \rng(a)$. But $\dom(a) = \bigcup (\square_{\leq 1} c \cap \bigcup a)$, which proves that domains of sets are sets. Now $F_{2,2,1} \upharpoonright c^2 : c^2 \rightarrow c^2$ is a set, and so is $(c^2 \times a) \cap F_{2,2,1}$. But the domain of this set is $a^{-1}$, and the domain of $a^{-1}$ is $\rng(a)$.

\eqref{ThmSepMap}: Let $f\in\VV$ be a map from $a$ to $b$, and let $c\subseteq b$ be closed. Then $f \cap (a\times c)$ is a set, too, and so is $f^{-1}[c] = \dom(f \cap (a\times c))$. Thus $f$ is continuous. Similarly, if $c\subseteq a$ is closed, then $f[c] = \rng(f\cap (c\times b))$ is a set and hence $f$ is closed.

\eqref{ThmSepLin}: Now let $a$ be linearly ordered by $\leq$. If $x\in a$, then $[x,\infty) = \rng((\{x\}\times a)\; \cap \leq)$ and $(\infty,x] = \dom((a\times \{x\})\; \cap \leq)$. Conversely assume that all intervals $[x,y]$ are sets. Then if $\fTup{x}{y}\in a^2 \setminus \leq$, that is, $x > y$. If there is a $z \in (y,x)$, then $(z,\infty)\times(-\infty,z)$ is a relatively open neighborhood of $\fTup{x}{y}$ disjoint from $\leq$. Otherwise, $(y,\infty)\times(-\infty,x)$ is one.

\eqref{ThmSepFinProd}: To show that the topologies on $a^n$ coincide, we only need to consider the case $n=2$; the rest follows by induction, because products of regular spaces are regular. Since $a$ is Hausdorff, we already know from Proposition \ref{ThmHausdorffSets} that the universal topology is at least as fine as the product topology, and it remains to prove the converse.

Let $b\subseteq a^2$ be a set. We will show that it is closed with respect to the product topology. Let $\fTup{x}{y} \in a^2 \setminus b$. Then $x\notin \dom (b \cap (a \times \{y\}))$, so by regularity, there is a closed neighborhood $u\ni x$ disjoint from that set. Thus $b\cap (a \times\{y\}) \cap (u\times a) = \emptyset$, that is, $y\notin \rng (b\cap (u\times a))$. Again by $T_3$, there is a closed neighborhood $v\ni y$ disjoint from that. Hence $b\cap (u\times v)=\emptyset$ and $u\times v$ is a neighborhood of $\fTup{x}{y}$ with respect to the product topology.

\eqref{ThmSepGPF}: We only have to prove $\Theta$ from Theorem \ref{ThmPosCompAxioms}: The statements about the permutations of triples are true because the topologies on products coincide. $\Delta \cap a$ is closed in $(\dom (a) \cup \rng(a))^2$, even with respect to the product topology, because every set is Hausdorff. $\mathbf{E}\cap a$ is a set by regularity: If $\fTup{x}{y}\in a\setminus \mathbf{E}$, then $x\notin y$, so $x$ and $y$ can be separated by disjoint $U\ni x$ and $V\supseteq y$ relatively open in $\dom(a) \cup \rng(a)$. $a \cap (U\times V)$ is a neighborhood of $\fTup{x}{y}$ disjoint from $\mathbf{E}$

It remains to show that $B = \{\fTup{x}{y}{\in} b \mid \Fa{z{\in}y} \fTupL{x,y,z}{\in}a\}$ is closed for every $a\in\VV$. Since
\[B \quad = \quad b \;\cap \; \{\fTup{x}{y}{\in} c^2 \mid \Fa{z{\in}y} \fTupL{x,y,z}{\in}a \cap c^3\}\text{,}\]
where $c=\dom(b) \cup \rng(b) \cup \bigcup \rng(b)$, we can wlog assume that $b= c^2$ and $a\subseteq c^3$, and prove that $B$ is a closed subset of $c^2$. Let $\fTup{x}{y} \in c^2 \setminus B$, that is, let $\Ex{z{\in}y} \fTupL{x,y,z}{\notin} a$. By \eqref{ThmSepFinProd} there exist relatively open neighborhoods $U$, $V$ and $W$ of $x$, $y$ and $z$ in $c$, such that $U\times V\times W$ is disjoint from $a$. But then $c \cap \lozenge W$ equals $c \setminus (\Atom \cup \square (c\setminus W))$ or $c \setminus (\Atom \cup \{\emptyset\} \cup \square (c\setminus W))$, depending on whether $\emptyset \in \VV$, so $c\cap\lozenge W$ is relatively open and hence $U\times\left(V \cap \lozenge W\right)$ is an open neighborhood of $\fTup{x}{y}$ in $c^2$ disjoint from $B$.
\end{proof}

Together with \eqref{ThmSepGPF}, Proposition \ref{LemPosTop} thus proves:

\begin{cor}
$\Th{GPK}^+_{(\infty)} + T_3 \;$ is equivalent to $\;\Th{TS}_{(\infty)}+(\Atom{=}\emptyset{\in}\VV)+\text{Union}+T_3$.
\end{cor}

\section{Uniformization}

Choice principles in the presence of a universal set are problematic. By Theorem \ref{ThmVinV}, for example, $\VV \in \VV$ implies that there is a perfect set and in particular that not every set is well-orderable. And in \cite{FortiHonsell1996, FortiHonsell1996add, Esser2000}, M. Forti, F. Honsell and O. Esser identified plenty of choice principles as inconsistent with positive set theory. On the other hand, many topological arguments rely on some kind of choice. The following \emphind{uniformization axiom} turns out to be consistent and yet have plenty of convenient topological implications, in particular with regard to compactness.

A \emphind{uniformization} of a relation $R\subseteq \VV^2$ is a function $F\subseteq R$ with $\dom(F) = \dom (R)$. The \emphind{uniformization axiom} states that we can simultaneously choose elements from a family of classes as long as it is indexed by a discrete set:
\begin{eqnarray*}
\text{\emph{Uniformization}} &\quad & \text{If $\dom(R)$ is a discrete set, $R$ has a uniformization.}
\end{eqnarray*}
Unless the relation is empty, its uniformization will be a set by the additivity axiom. Therefore the uniformization axiom can be expressed with at most one universal and no existential quantification over classes, and thus still be equivalently formulated in a first-order way, using axiom schemes. Let us denote by $\Th{ESU}$ resp. $\Th{TSU}$ essential resp. topological set theory with uniformization.

In these theories, at least all discrete sets are well-orderable. The following proof goes back to S. Fujii and T. Nogura (\cite{Fujii1999}). We call $f:\square a \rightarrow a$ a \emphind{choice function} if $f(b)\in b$ for every $b\in\square a$.

\begin{prp}[$\Th{ESU}$]\label{ThmUnif}
A set $a$ is well-orderable iff it is Hausdorff and there exists a continuous choice function $f:\square a \rightarrow a$, such that $b\setminus \{f(b)\}$ is closed for all $b$.

In particular, every discrete set is well-orderable and in bijection to $\kappa^\oplus$ for some cardinal $\kappa$.
\end{prp}

\begin{proof}
If $a$ is well-ordered, we only have to define $F(b)= \min(b)$. In a well-order, the minimal element is always isolated, so $b\setminus F(b)$ is in fact closed. To show that $F$ is a set, let $c\subseteq a$ be closed. Then the preimage of $c$ consists of all nonempty subsets of $a$ whose minimal element is in $c$. Assume $b\notin F^{-1}[c]$, that is $F(b)\notin c$. Then
\[(\square a \;\cap \; \lozenge ((-\infty, F(b)] \cap c)) \;\cup\; \square [F(b)+1,\infty)\]
is a closed superset of $F^{-1}[c]$ omitting $b$, where by $F(b)+1$ we denote the successor of $F(b)$, and if $F(b)$ is the maximal element, we consider the right part of the union to be empty. Hence $F^{-1}[c]$ is in fact closed, proving that $F$ is continuous and a set.

For the converse, assume now that $f$ is a continuous choice function. A set $p\subseteq \square a$ is an \emph{approximation} if:
\begin{myitemize}
\item $a\in p$
\item $p$ is well-ordered by reverse inclusion $\supseteq$.
\item For every nonempty proper initial segment $Q\subset p$, we have $\bigcap Q \in p$.
\item For every non-maximal $b\in p$, we have $b\setminus \{f(b)\} \in p$.
\end{myitemize}
We show that two approximations $p$ and $q$ are always initial segments of one another, so they are well-ordered by inclusion: Let $Q$ be the initial segment they have in common. Since both contain $b=\bigcap Q$, that intersection must be in $Q$ and hence the maximal element of $Q$. If $b$ is not the maximum of either $p$ or $q$, both contain $b\setminus \{f(b)\}$, which is a contradiction because that is not in $Q$.

Thus the union $P$ of all approximations is well-ordered. Assume $\bigcap P$ has more than one element. Then $P \cup \{\bigcap P, \bigcap P \setminus f(\bigcap P)\}$ were an approximation strictly larger than $P$. Thus $\bigcap P$ is empty or a singleton. Since there is no infinite descending chain, and for every bounded ascending chain $Q\subseteq P$, we have $\bigcap Q \in P$, $P$ is closed, so $P\in\VV$. Also, $\supseteq$ is a set-well-order on $P$. Thus $a$ is also set-well-orderable, because $f\upharpoonright P$ is a continuous bijection onto $a$:

Firstly, it is injective, because after the first $b$ with $f(b)=x$, $x$ is omitted. Secondly, it is surjective, because if $b\in P$ is the first element not containing $x$, it cannot be the intersection of its predecessors and thus has to be of the form $b=c\setminus \{f(c)\}$. Hence $x=f(c)$. If $x$ is a member of every element of $P$, then $\bigcap P = \{x\} \in P$ and $x = f(\{x\})$.

Now let $a$ be a discrete set. We only have to prove that a continuous choice function $f:\square a \rightarrow a$ exists. In fact, any choice function will do, since $\square a$ is discrete and hence every function on $\square a$ is continuous. And the existence of such a function follows from the uniformization axiom, applied to the relation $R\subseteq \square d \times d$ defined by: $xRy$ iff $y \in x$.

It follows that every discrete set $a$ is well-orderable. Therefore, it is comparable in length to $\On$. If an initial segment of $a$ were in bijection to $\On$, then as the image of a discrete set, $\On$ would be a set. Hence $a$ must be in bijection to a proper initial segment $\alpha^\oplus$ of $\On$. If $\kappa$ is the cardinality of $\alpha$, there is a bijection between $\kappa^\oplus$ and $\alpha^\oplus$. Composing these bijections proves the claim.
\end{proof}

It follows that there exists an infinite discrete set iff $\omega\in \On$. The uniformization axiom also allows us to define for every infinite cardinal $\kappa$ a cardinal $2^\kappa$, namely the least ordinal in bijection to $\square \kappa^\oplus$. Proposition \ref{ThmUnif} then shows that, just as in $\Th{ZFC}$, $\On$ is not only a weak but even a strong limit.

Like the axiom of choice, the uniformization axiom could be stated in terms of products. Of course, it only speaks of products of $\Discrete$-few factors at first, but surprisingly it even has implications for larger products as long as the factors are indexed by a $\Discrete$-compact well-ordered set. $\Discrete$-compactness for a well-ordered set just means that no subclass of cofinality $\geq \On$ is closed.

\begin{prp}[$\Th{ESU}+T_3+\text{Union}$]
Let $w$ be a $\Discrete$-compact well-ordered set, $a\in\VV$ and $a_x\subseteq a$ nonempty for every $x\in w_I$. Then the product $\prod_{x\in w_I} a_x$ is nonempty.
\end{prp}

\begin{proof}
Recall that the product is defined as:
\[\prod_{x\in w_I} a_x \quad = \quad \left\{F \cup (w'{\times}a) \;\mid\; F:w_I \rightarrow \VV,\; \Fa{x} F(x)\in a_x \right\}\]
We do induction on the length of $w$ and we have to distinguish three cases:

{\bf Case 1:} If $w$ has no greatest element, its cofinality must be $\Discrete$-small or else it would not be $\Discrete$-compact Hausdorff. So let $\fSeq{y_\alpha}{\alpha<\kappa}$ be a cofinal strictly increasing sequence. Using the induction hypothesis and the uniformization axiom, choose for every $\alpha<\kappa$ an element
\[f_\alpha \quad \in \quad \prod_{x\in]y_\alpha,y_{\alpha+1}]_I} a_x\]
Then the union of the $f_\alpha$ is an element of $\prod_{x\in w_I} a_x$.

{\bf Case 2:} Assume that $w$ has a greatest element $p$ and that $w\setminus \{p\}$ is a set. Then this is still $\Discrete$-compact Hausdorff and hence the induction hypothesis applies, so there is an element $f:w\setminus \{p\} \rightarrow a$ of the product missing the last dimension. For any $y\in a_p$, the set $f\cup \fTup{p}{y}$ is in $\prod_{x\in w_I} a_x$.

{\bf Case 3:} Finally assume that $w$ has a greatest element $p$ and that $w\setminus \{p\}$ is not a set. By the induction hypothesis,
\[P_y \;=\; \prod_{x\in [-\infty,y]_I} a_x\]
is a nonempty set for every $y<p$. The union $Q=\bigcup_{y<p} P_y$ is not a set, because otherwise its domain $\dom \left(\bigcup Q\right) = w\setminus\{p\}$ would also be a set. But since $Q\subseteq \square(w \times a)$, it does have a closure which is a set, and this closure must have an element $g$ with $p\in \dom(g)$. We will show that $f=g \cup (w'\times a)$ witnesses the claim, that is, $f\in \prod_{x\in w_I} a_x$.

If $z\in w_I$, then $g$ is not in the closure of $\bigcup_{y<z} P_y$, because that is a subclass of the set $\square ((-\infty,z]\times a)$. Thus $g$ is in the closure of $\bigcup_{z\leq y<p} P_y$, which is a subclass of:
\[M_z \quad = \quad \square (w\times a) \;\cap\; \{r \mid r \cap (\{z\}\times a) \in \square_{\leq 1} a_z \}\]
If we can show that $M_z$ is closed, we can deduce that $g\in M_z$ for every $z\in w_I$ and therefore $g\upharpoonright w_I = f \upharpoonright w_I$ is a function from $w_I$ to $a$ with $f(x)\in a_x$ for all $x\in w_I$. Thus $f$ is indeed an element of the product.

To prove that $M_z$ is closed in $\square (w\times a) \cap \lozenge (\{z\}\times a_z)$, assume $r$ is an element of the latter but not of the former. Then there are distinct $x_1,x_2\in a_z$, such that $\fTup{z}{x_1}, \fTup{z}{x_2} \in z$. Since $a_z$ is Hausdorff, there are $u_1$ and $u_2$, such that $x_1\notin u_1$, $x_2\notin u_2$ and $u_1\cup u_2 = a$, and
\[\square \left((w\setminus \{z\}) \times a \;\cup\; \{z\}\times u_1\right) \quad\cup\quad \square \left((w\setminus \{z\}) \times a \;\cup\; \{z\}\times u_2\right) \]
is a closed superset of $M_z$ omitting $r$.
\end{proof}

Some of the known models of topological set theory are ultrametrizable, which in the presence of the uniformization axiom is a very strong topological property. A set $a$ is \emph{ultrametrizable}\index{ultrametrizable set} if there is a decreasing sequence $\fSeq{\sim_\alpha}{\alpha{\in}\On}$ of equivalence relations on $a$ such that $\bigcap_\alpha \sim_\alpha = \Delta_a$ and the $\alpha$-\emphind{ball}s $[x]_\alpha = \{y\mid x\sim_\alpha y\}$ for $x\in a$ and $\alpha\in\On$ are a base of the natural topology on $a$ in the sense of open classes, that is, the relatively open classes $U\subseteq a$ are exactly the unions of balls. If that is the case, the $\alpha$-balls partition $a$ into clopen sets for every $\alpha$.

\begin{prp}[$\Th{ESU}$]
Every ultrametrizable set is a $\Discrete$-compact linearly orderable set.
\end{prp}

\begin{proof}
For every $\alpha\in\On$, the class $C_\alpha$ of all $\alpha$-balls is a subclass of $\square a$. If $b\in \square a$ and $x\in b$, then $\lozenge [x]_\alpha$ is a neighborhood of $b$ in $\square a$ which contains only one element of $C_\alpha$, namely $[x]_\alpha$. Hence $C_\alpha$ has no accumulation points and is therefore a discrete set. That means there are only $\Discrete$-few $\alpha$-balls for every $\alpha\in\On$.

Now let $A\subseteq \square a$ and $\bigcap A=\emptyset$. For each $\alpha$, let $B_\alpha$ be the union of all $\alpha$-balls which intersect every element of $A$. Then $\bigcap_{\alpha} B_\alpha = \emptyset$ and every $B_\alpha$ is closed.

Assume that all $B_\alpha$ are nonempty. Then for every $\alpha$ all but $\Discrete$-few members of the sequence $\fSeq{B_\alpha}{\alpha{\in}\On}$ are elements of the closed set $\square B_\alpha$, so every accumulation point must be in $\bigcap_{\alpha\in\On} \square B_\alpha$, which is empty. Thus $\{B_\alpha \mid \alpha{\in}\On\}$ has no accumulation point and is a discrete subset of $\square B_0$. Hence it is $\Discrete$-small, which means that the sequence $\fSeq{B_\alpha}{\alpha{\in}\On}$ is eventually constant, a contradiction.

Therefore there is a $B_\alpha$ which is empty, and by definition every $\alpha$-ball is disjoint from some element of $A$. Since there are only $\Discrete$-few $\alpha$-balls, the uniformization axiom allows us to choose for every $\alpha$-ball $[x]_\alpha$ an element $c_{[x]_\alpha} \in A$ disjoint from $[x]_\alpha$. The set of these $c_{[x]_\alpha}$ is discrete and has an empty intersection. This concludes the proof of the $\Discrete$-compactness.

Since it is discrete, the set $C_\alpha$ can be linearly ordered and there are only $\Discrete$-few such linear orders for every $\alpha$. If $L$ is a linear order on $C_\alpha$, let $R_L$ be the partial order relation on $a$ defined by $xR_L y$ iff $[x]_\alpha L [y]_\alpha$. $R_L$ is a set because it is the union of $\Discrete$-few sets of the form $[x]_\alpha \times [y]_\alpha$. Let $S_\alpha$ be the set of all such $R_L$. The sequence $\fSeq{S_\alpha}{\alpha{\in}\On}$ can only be eventually constant if $a$ is discrete, in which case it is linearly orderable anyway. If $a$ is not discrete, however, $S = \bigcup_\alpha S_\alpha$ must be $\Discrete$-large and therefore have an accumulation point $\leq$ in $\square a^2$. Because each $S_\alpha$ is $\Discrete$-small, $\leq$ is in the closure of every $\bigcup_{\beta>\alpha} S_\beta$. For $x,y\in a$, let
\[t_{\alpha,x,y} \quad=\quad \square (a^2 \setminus ([y]_\alpha \times [x]_\alpha)) \;\cap\; \lozenge \{\fTup{x}{y}\}\text{.}\]
We will show that $\leq$ is a linear order on $a$:

Assume $x\neq y$. Then there is an $\alpha$ such that $x \nsim_\alpha y$. Every element of $\bigcup_{\beta>\alpha} S_\beta$ assigns an order to $[x]_\alpha$ and $[y]_\alpha$, so it is in exactly one of the disjoint closed sets $t_{\alpha,x,y}$ and $t_{\alpha,y,x}$.
Therefore the same must be true of $\leq$, so we have $x \leq y$ iff not $y\leq x$. This proves antisymmetry and totality.

If $x\leq y \leq z$ and $x,y,z$ are distinct, then there is an $\alpha$ such that $x\nsim_\alpha y\nsim_\alpha z \nsim_\alpha x$. Then $\leq$ is in the closure of neither $t_{\alpha,y,x}$ nor $t_{\alpha,z,y}$, and must therefore be in the closure of
\[\bigcup_{\beta>\alpha} S_\beta \; \cap \; t_{\alpha,x,y} \;\cap \; t_{\alpha,y,z}\text{,}\]
which is a subset of $t_{\alpha,x,z}$, because every element of $S$ is transitive. It follows that $\leq$ is also in $t_{\alpha,x,z}$ and thus $x\leq z$, proving transitivity.

Finally, $\leq$ is reflexive because for every $x\in a$, all of $S$ lies in the set $\square a^2 \;\cap\; \lozenge \{\fTup{x}{x}\}$.
\end{proof}

%Ultrametrizability limits the length of set-well-orders: Their initial segments all have cofinality at most $\On$.

Another consequence of the uniformization axiom is the following law of distributivity:

\begin{lem}[$\Th{ESU}$]\label{LemDistributivity}
If $d$ is discrete and for each $i\in d$, $J_i$ is a nonempty class, then
\[\bigcup_{i\in d} \; \bigcap_{j\in J_i} \; j \quad = \quad  \bigcap_{f \in \prod_{i\in d} J_i\:} \; \bigcup_{i\in I} \; f(i)\text{.}\]
\end{lem}

\begin{proof}
If $x$ is in the set on the left, there exists an $i\in d$ such that $x$ is an element of every $j\in J_i$. Thus for every function $f$ in the product, $x\in f(i)$. Hence $x$ is an element of the right hand side.

Conversely, assume that $x$ is not in the set on the left, that is, for every $i\in d$, there is a $j\in J_i$ such that $x\notin j$. Let $f$ be a uniformization of the relation $R = \{\fTup{i}{j} \mid i \in d, \; x \notin j \in J_i\}$. Then $x\notin \bigcup_{i\in I} f(i)$.
\end{proof}

It implies that we can work with subbases in the familiar way. Let us call $\calK$ \emph{regular}\index{regular class} if every union of $\calK$-few $\calK$-small sets is $\calK$-small again. Then in particular $\Discrete$ is regular.

\begin{lem}[$\Th{ESU}$]\label{LemSubbase}
Let $\calK\subseteq \Discrete$ and let $B$ be a $\calK$-subbase of a topology $T$ such that the union of $\calK$-few elements of $B$ always is an intersection of elements of $B$. Then $B$ is a base of $T$.
\end{lem}

\begin{proof}
We only have to prove that the intersections of elements of $B$ are closed with respect to $\calK$-small unions and therefore constitute a $\calK$-topology. But if $I$ is $\calK$-small, and each $\fSeq{b_{i,j}}{j\in J_i}$ is a family in $B$, we have
\[\bigcup_{i\in I} \; \bigcap_{j\in J_i} \; b_{i,j} \quad = \quad  \bigcap_{f \in \prod_{i\in I} J_i\:} \; \bigcup_{i\in I}\; b_{i,f(i)}\]
by Lemma \ref{LemDistributivity}, and every $\calK$-small union $\bigcup_{i\in I} b_{i,f(i)}$ is an element of $B$ again.
\end{proof}

Thus if $\calK$ is regular and $S$ is a $\calK$-subbase of $T$, the class of all $\calK$-small unions of elements of $S$ is a base of $T$. Since $\bigcup_i \lozenge a_i = \lozenge \bigcup_i a_i$, the sets of the following form constitute a base of the exponential $\calK$-topology:
\[\lozenge_T a \; \cup \; \bigcup_{i{\in}I} \square_T b_i\text{,}\]
where $I$ is $\calK$-small and $a,b_i \in T$ for all $i\in I$. As that is sometimes more intuitive, we also use open classes in our arguments instead of closed sets. By setting $U = \complement a$ and $V_i = \complement b_i$, we obtain that every open class is a union of classes of the following form:
\[\square_T U \: \cap \: \bigcap_{i\in I} \lozenge_T V_i\]
That is, these constitute a base in the sense of \emph{open} classes. Since $\square U = \square U \cap \lozenge U$, the class $U$ can always be assumed to be the union of the $V_i$.

Lemma \ref{LemSubbase} also implies that given a class $B$, the weak comprehension principle suffices to prove the existence of the topology $\calK$-generated by $B$: A set $c$ is closed iff for every $x\in \complement a$, there is a discrete family $(b_i)_{i\in I}$ in $B$, such that $c\subseteq \bigcup_i b_i$ and $x\notin \bigcup_i b_i$. In particular, the $\calK$-topology of $\ExSpc_\calK (X)$ exists (as a class) whenever the topology of $X$ is a set.

\begin{lem}[$\Th{ESU}$]\label{LemSepExp}
Let $\calK$ be regular and $X$ a $\calK$-topological $T_0$-space.
\begin{enumerate}
\item\label{ThmExpSepT1} If $X$ is $T_1$, then $\ExSpc_\calK(X)$ is $T_1$ (but not necessarily conversely).
\item\label{ThmExpSepT2T3} $X$ is $T_3$ iff $\ExSpc_\calK(X)$ is $T_2$.
\item\label{ThmExpSepT3T4} $X$ is $T_4$ iff $\ExSpc_\calK(X)$ is $T_3$.
\end{enumerate}
\end{lem}

\begin{proof}
In this proof we use $\square$ and $\lozenge$ with respect to $X$, not the universe, so if $T$ is the topology of $X$, we set $\square a = \square_T a$ and $\lozenge a = \lozenge_T a$.

\eqref{ThmExpSepT1}: For $a\in\ExSpc_\calK(X)$, the singleton $\{a\} = \square a \cap \bigcap_{x\in a} \lozenge \{x\}$ is closed in $\ExSpc_\calK(X)$.

(As a counterexample to the converse consider the case where $\calK=\kappa$ is a regular cardinal number and $X = (\kappa+ 1)^\oplus$, with the $\kappa$-topology generated by the singletons $\{\alpha\}$ for $\alpha<\kappa$. This is not $T_1$, because $\{\kappa\}$ is not closed, but it is clearly $T_0$. We show that its exponential $\kappa$-topology is $T_1$: Let $a\in \ExSpc_\calK(X)$. Then either $a\subseteq \kappa$ is small or $a = X$.

In the first case, $\{a\} =  \square a \cap \bigcap_{x\in a} \lozenge \{x\}$ is closed. In the second case, $\{a\} = \{X\} = \bigcap_{x\in \kappa} \lozenge \{x\}$ is also closed.)

\eqref{ThmExpSepT2T3}: ($\Rightarrow$) Let $a,b \in \ExSpc_\calK(X)$ be distinct, wlog $x\in b\setminus a$. Then there are disjoint open $U,V\subseteq X$ separating $x$ from $a$. Hence $\lozenge U$ and $\square V$ separate $b$ from $a$.

($\Leftarrow$) Firstly, we have to show that $X$ is $T_1$. Assume that $\{y\}$ is not closed, so there exists some other $x\in\fCl(\{y\})$, and by $T_0$, $y$ is not in the closure of $x$, so $\fCl(\{x\}) \subset \fCl(\{y\})$. The two closures can be separated by open base classes $\square U \cap \bigcap_{i} \lozenge U_i$ and $\square V \cap \bigcap_{j} \lozenge V_j$ of $\ExSpc_\calK(X)$, whose intersection $\square (U\cap V) \cap \bigcap_{i} \lozenge U_i \cap \bigcap_{j} \lozenge V_j$ is emtpy. Hence there either exists a $U_i$ disjoint from $V$ -- which is impossible because $\fCl(\{x\})\in \square V \cap \bigcap_{i} \lozenge U_i$ --, or there is a $V_j$ disjoint from $U$: But since $V_j \cap \fCl(\{y\}) \neq \emptyset$, we have $y\in V_j$. Hence $y\notin U \ni x$, contradicting the assumption that $x$ is in the closure of $y$.

Now let $x\notin a$. Then $a$ and $b = \{x\} \cup a$ can be separated by open base classes $\square U \cap \bigcap_{i} \lozenge U_i$ and $\square V \cap \bigcap_{j} \lozenge V_j$ of $\ExSpc_\calK(X)$, whose intersection $\square (U\cap V) \cap \bigcap_{i} \lozenge U_i \cap \bigcap_{j} \lozenge V_j$ is emtpy. Hence there either exists a $U_i$ disjoint from $V$ -- which is impossible because $a\in \square V \cap \bigcap_{i} \lozenge U_i$ --, or there is a $V_j$ disjoint from $U$: Then $V_j$ and $U$ separate $x$ from $a$, because $b$ meets $V_j$ and $a$ does not, so $x\in V_j$.

\eqref{ThmExpSepT3T4}: In both directions, the $T_1$ property follows from the previous points.

($\Rightarrow$) Let $a\notin c$, $a\subseteq X$ closed and $c\subseteq \ExSpc_\calK(X)$ closed. Wlog\footnote{To verify that a space $X$ is $T_3$ it suffices to separate each point $x$ from each subbase set $b$ not containing $x$: Firstly, the $\calK$-small unions of subbase sets $b$ are a base, so if $x$ is not in a $\calK$-small union $\bigcup_i b_i$, it can be separated with $U_i, V_i$ from every $b_i$, and $\bigcap U_i, \bigcup V_i$ separate $x$ from the union. This shows that $x$ can then be separated from each base set. Secondly, every closed set is an intersection $\bigcap_i b_i$ of base sets $b_i$, and if $x$ is not in that intersection, there is an $i$ with $x\notin b_i$ and if $U_i, V_i$ separate $x$ from $b_i$, they also separate $x$ from $\bigcap_i b_i$.}
let $c$ be of the form $\square b$ or $\lozenge b$ with closed $b\subseteq X$. In the first case, $a\nsubseteq b$, so let $U,V$ separate some $x\in a\setminus b$ from $b$. Then $\lozenge U,\square V$ separate $\{a\},c$. In the second case, $a\cap b= \emptyset$, so let $U,V$ separate them. Then $\square U,\lozenge V$ separate $\{a\},c$.

($\Leftarrow$) Now  let $\ExSpc_\calK(X)$ be $T_3$ and let $a,b\subseteq X$ be closed, nonempty and disjoint. Then $\{a\}$ and $\lozenge b$ are disjoint and can be separated by disjoint open $U,V\subseteq \ExSpc_\calK(X)$. $U$ can be assumed to be an open base class, so $U= \square W \cap \bigcap_{i} \lozenge W_i$. We claim that $\fCl(W)\cap b= \emptyset$, which proves the normality of $X$. So assume that there exists $x\in \fCl(W)\cap b$. Then $a \cup \{x\}\in \lozenge b$, so one of the open base classes $\square Z \cap \bigcap_{j} \lozenge Z_j$ constituting $V$ must contain $a \cup \{x\}$. That means that either one of the $Z_j$ must be disjoint from $W$ -- which is impossible because $x\in Z_j$ -- or one of the $W_i$ must be disjoint from $Z$ -- which also cannot be the case, because all $W_i$ intersect $a$ and $a\subseteq Z$.
\end{proof}

\section{Compactness}

Hyperuniverses are $\Discrete$-compact Hausdorff spaces, so $\Discrete$-compactness is another natural axiom to consider. In the case $\VV\notin \VV$, the corresponding statement would be that every set is $\Discrete$-compact (note that this is another axiom provable in $\Th{ZFC}$), but if $\VV\in\VV$, this is equivalent to $\VV$ being $\Discrete$-compact. And in fact, $\Th{TSU}$ with a $\Discrete$-compact Hausdorff $\VV$ implies most of the additional axioms we have looked at so far, including the separation properties and the union axiom:

Let $a\subseteq\TT$ and $x\notin \bigcup a$. Then for every $y\in a$, there is a $b$ such that $y\subseteq \fInt(b)$ and $x\notin b$. The sets $\square \fInt(b)$ then cover $a$ and by $\Discrete$-compact Hausdorffness, a discrete subfamily also does. But then the union of these $b$ is a superset of $\bigcup a$ not containing $x$.

Another consequence of global $\Discrete$-compactness is that most naturally occurring topologies coincide: Point \eqref{ThmDCptNatTop} of the following theorem not only applies to hyperspaces $\square a$, but also to products, order topologies and others. If the class of atoms is closed and unions of sets are sets, this even characterizes compactness (note that these two assumptions are only used in \eqref{ThmDCptExp} $\Rightarrow$ \eqref{ThmDCptDCpt}):

\begin{thm}[$\Th{ESU}+T_2+\text{Union}$] If $\Atom$ is $\TT$-closed, the following statements are equivalent:
\begin{myenumerate}
\item\label{ThmDCptDCpt} Every set is $\Discrete$-compact, that is: If $\bigcap A {=}\emptyset$, there is a discrete $d{\subseteq}A$ with $\bigcap d {=} \emptyset$.
\item\label{ThmDCptNatTop} Every Hausdorff $\Discrete$-topology $T\in\VV$ equals the natural topology: $T = \square \bigcup T$
\item\label{ThmDCptExp} For every set $a$, the exponential $\Discrete$-topology on $\square a$ equals the natural topology.
\end{myenumerate}
\end{thm}

\begin{proof}
\eqref{ThmDCptDCpt} $\Rightarrow$ \eqref{ThmDCptNatTop}: Let $A=\bigcup T$. Since $A$ is $T$-closed in $A$, $A\in T$ and thus $A\in \VV$. By definition, $T\subseteq \square A$. For the converse, we have to verify that each $b\in\square A$ is $T$-closed, so let $y\in A\setminus b$. Consider the class $C$ of all $u\in T$, such that there is a $v\in T$ with $u\cup v=A$ and $y\notin v$. By the Hausdorff axiom, for every $x\in b$ there is a $u\in C$ omitting $x$, so $b\cap \bigcap C = \emptyset$. By $\Discrete$-compactness, there is a discrete $d\subseteq C$ with $b \cap \bigcap d = \emptyset$. By definition of $C$, $y\in \fInt_T(u)$ for every $u$, and since $d$ is discrete, the intersection $\bigcap_{u\in d} \fInt_T(u)$ is open. Therefore, every $y\notin b$ has a $T$-open neighborhood disjoint from $b$.

\eqref{ThmDCptNatTop} $\Rightarrow$ \eqref{ThmDCptExp} is trivial, because as a $\Discrete$-compact Hausdorff $\Discrete$-topological space, $a$ is $T_3$ and hence $\square a$ is Hausdorff by Lemma \ref{LemSepExp}.

\eqref{ThmDCptExp} $\Rightarrow$ \eqref{ThmDCptDCpt}: Lemma \ref{LemSepExp} also implies that if $\square \square a$ is $T_2$, then $\square a$ is $T_3$ and $a$ is $T_4$, so it follows from the Hausdorff axiom that every set is normal.

Finally, we can prove $\Discrete$-compactness. Let $A\subseteq \square a$, $\bigcap A = \emptyset$ and let $c=\fCl(A)$. Then $\bigcap c = \emptyset$. Since every set is regular and $\Atom$ is closed, the positive specification principle holds. Therefore
\[B \quad = \quad \left\{ b{\in}\square c \mid \bigcap b \neq \emptyset\right\} \quad = \quad \{b{\in}\square c \mid \Ex{x} \Fa{y{\in}b} x{\in} y\}\]
is a closed subset of $\square c$ not containing $c$. In particular, there is an open base class
\[\square U \;\cap\; \bigcap_{i\in I} \lozenge V_i\]
of the space $\square c$ containing $c$ which is disjoint from $B$. Every $U\cap V_i$ is a relatively open subset of $c$, so there is an $x_i \in A\cap U \cap V_i$, because $A$ is dense in $c$. The set $\{x_i \mid i{\in}I\}$ --  and here we used the uniformization axiom -- then is a discrete subcocover of $A$.
\end{proof}

\clearpage
%\nocite{FortiHonsell1996}
\bibliography{diss}

\end{document}